\documentclass[10pt]{article}

\usepackage{fullpage}
\usepackage{amsmath,amssymb,amsthm,hyperref,bm}
\usepackage[noadjust]{cite}
\usepackage{tikz, tikz-cd}
\usetikzlibrary{positioning}
\usepackage{mathtools}
\usepackage{verbatim}
\usepackage{multirow}
\usepackage{array}
\usepackage{enumitem} 
\usepackage[english]{babel}

\hypersetup{hidelinks} 

\makeatletter
\renewcommand*\env@matrix[1][\arraystretch]{%
  \edef\arraystretch{#1}%
  \hskip -\arraycolsep
  \let\@ifnextchar\new@ifnextchar
  \array{*\c@MaxMatrixCols c}}
\makeatother

\theoremstyle{definition}
\newtheorem{definition}{Definition}[section]
\newtheorem{remark}[definition]{Remark}

\newtheorem{example}[definition]{Example}

\newtheorem{notation}[definition]{Notation}

\theoremstyle{plain}
\newtheorem{theorem}[definition]{Theorem}
\newtheorem{lemma}[definition]{Lemma}
\newtheorem{corollary}[definition]{Corollary}
\newtheorem{proposition}[definition]{Proposition}

\def\tet{\boxtimes}
\def\L{L(\mathfrak{sl}_2)^+}
\def\sl2{\mathfrak{sl}_2}

\def\bx{\boldsymbol{x}}
\def\by{\boldsymbol{y}}
\def\bz{\boldsymbol{z}}

\begin{document}

\title{\bf The standard generators of the tetrahedron\\ algebra and their look-alikes}

\author {
Jae-Ho Lee\thanks{Department of Mathematics and Statistics, University of North Florida, Jacksonville, FL 32224, USA
}}
\date{}

\maketitle
\newcommand\blfootnote[1]{%
\begingroup
\renewcommand\thefootnote{}\footnote{#1}%
\addtocounter{footnote}{-1}%
\endgroup}
\blfootnote{E-mail address:  \texttt{jaeho.lee@unf.edu}}

\begin{abstract}
\noindent
The tetrahedron algebra $\boxtimes$ is an infinite-dimensional Lie algebra defined by generators $\{x_{ij} \mid i, j \in \{0, 1, 2, 3\}, i \neq j\}$ and some relations, including the Dolan-Grady relations. 
These twelve generators are called standard. 
We introduce a type of element in $\boxtimes$ that ``looks like'' a standard generator.
For mutually distinct $h, i, j, k \in \{0, 1, 2, 3\}$, consider the standard generator $x_{ij}$ of $\boxtimes$. 
An element $\xi \in \boxtimes$ is called $x_{ij}$-like whenever both (i) $\xi$ commutes with $x_{ij}$; (ii) $\xi$ and $x_{hk}$ satisfy a Dolan-Grady relation. 
Pick mutually distinct $i,j,k \in \{0,1,2,3\}$.
In our main result, we find an attractive basis for $\boxtimes$ with the property that every basis element is either $x_{ij}$-like or $x_{jk}$-like or $x_{ki}$-like.
We discuss this basis from multiple points of view.


\end{abstract}
\section{Introduction}\label{Intro}
This paper is about a certain Lie algebra, called the tetrahedron algebra. 
This Lie algebra, denoted by $\tet$, was introduced by Hartwig and Terwilliger to provide a presentation of the three-point $\sl2$ loop algebra $\L$ \cite{2007HarTer}.
They defined $\tet$ by generators and relations, and displayed a Lie algebra isomorphism $\tet \to \L$.
We now recall the definition of $\tet$.
Throughout this paper, let $\mathbb{F}$ denote a field with characteristic $0$ and define the set $\mathbb{I}=\{0, 1, 2, 3\}$.

\begin{definition}[{\cite[Definition 1.1]{2007HarTer}}]\label{Def:Tet-alg}
Let $\tet$ denote the Lie algebra over $\mathbb{F}$ that has generators
\begin{equation}\label{standard generators tet}
	\{x_{ij} \mid i,j \in \mathbb{I}, \ i \ne j\}, 
\end{equation}
and the following relations:
\begin{itemize}
	\item[(i)] For distinct $i,j \in \mathbb{I}$,
	\begin{equation}\label{def:tet rel(1)}
		x_{ij} + x_{ji} = 0.
	\end{equation}
	\item[(ii)] For mutually distinct $h,i,j \in \mathbb{I}$,
	\begin{equation}
		[x_{hi}, x_{ij}] = 2x_{hi}+2x_{ij}.
	\end{equation}
	\item[(iii)] For mutually distinct $h,i,j, k\in \mathbb{I}$,
	\begin{equation}\label{Dolan-Grady rels}
		[x_{hi}, [ x_{hi}, [x_{hi}, x_{jk}]]] = 4[x_{hi}, x_{jk}].
	\end{equation}
\end{itemize}
We call $\tet$ the \emph{tetrahedron algebra}.
We call the elements $x_{ij}$ in \eqref{standard generators tet} the \emph{standard generators} for $\tet$.
\end{definition}

\noindent
We observe that $\tet$ has twelve standard generators.
The relations \eqref{Dolan-Grady rels} are known as the \emph{Dolan-Grady} relations.

\begin{remark}
For mutually distinct $i, j, k\in \mathbb{I}$, the standard generators $x_{ij}$, $x_{jk}$, $x_{ki}$ form a basis for a subalgebra of $\tet$ that is isomorphic to the Lie algebra $\mathfrak{sl}_2$; cf. \cite[Corollary 12.4]{2007HarTer}.
\end{remark}

We have some comments on the symmetry of the Lie algebra $\tet$.
Consider the symmetric group $S_4$ on $\mathbb{I}$.
The group $S_4$ acts on the standard generators for $\tet$ by permuting the indices. 
This action induces a group homomorphism $S_4 \to \operatorname{Aut}(\tet)$.
This homomorphism is injective \cite[Corollary 12.7]{2007HarTer}.
We identify each element of $S_4$ with its image under this group homomorphism. 
For example, let $\prime$ denote the permutation $(123) \in S_4$. 
Regarding $\prime$ as an element in $\operatorname{Aut}(\tet)$, we have
\begin{align}
	& x'_{01} = x_{02}, \qquad \qquad x'_{02} = x_{03}, \qquad \qquad x'_{03} = x_{01}, \label{S4:(123)_1}\\
	& x'_{12} = x_{23}, \qquad \qquad x'_{23} = x_{31}, \qquad \qquad x'_{31} = x_{12}. \label{S4:(123)_2}
\end{align}

We have some comments about the history of $\tet$.
Over the years, $\tet$ has been studied in various contexts, including tridiagonal pairs \cite{2007Har, 2007ItoTerLAA, 2008ItoTerCA}, the universal central extension of $\L$ \cite{2007BenTer}, distance-regular graphs \cite{2014MorPas}, the special orthogonal Lie algebra $\mathfrak{so}_4$ \cite{2022MorLAA}, and the Onsager Lie algebra \cite{2007HarTer}.
This Onsager Lie algebra and its connection to $\tet$ play a role in the present paper.
Before discussing this connection in detail, we recall some background on the Onsager Lie algebra.

In \cite{1944Ons}, Lars Onsager introduced an infinite-dimensional Lie algebra in his study of the two-dimensional Ising model in a zero magnetic field.
This Lie algebra, now called the Onsager Lie algebra, has made a seminal contribution to the theory of exactly solved models in statistical mechanics.
 In \cite{1989Per}, Perk gave a presentation of the Onsager Lie algebra involving generators and relations.
We now recall this presentation.

\begin{definition}[\cite{1944Ons,1989Per}]
Let ${O}$ denote the Lie algebra over $\mathbb{F}$ with generators $A, B$ and relations
\begin{align}
	[A, [A, [A, B]]] & = 4[A,B], \label{OnsagerDG(1)}\\
	[B, [B, [B, A]]] & = 4[B,A]. \label{OnsagerDG(2)}
\end{align}
We call ${O}$ the \emph{Onsager Lie algebra}. 
We call $A, B$ the \emph{standard generators} for ${O}$.
Note that the relations \eqref{OnsagerDG(1)}, \eqref{OnsagerDG(2)} are the Dolan-Grady relations.
\end{definition}

\noindent
We now clarify the connection between $O$ and $\tet$.
For mutually distinct $h,i,j,k \in \mathbb{I}$, there exists an injective Lie algebra homomorphism ${O} \to \tet$ that sends
\begin{equation}
	A \ \longmapsto \ x_{hi}, \qquad \qquad 
	B \ \longmapsto \ x_{jk};
\end{equation}
cf. \cite[Proposition 4.7, Corollary 12.2]{2007HarTer}.
We call the image of $O$ under this homomorphism an \emph{Onsager subalgebra of $\tet$}.
The Lie algebra $\tet$ has three Onsager subalgebras.
Moreover, the $\mathbb{F}$-vector space $\tet$ is a direct sum of these three Onsager subalgebras; cf. \cite[Proposition 7.8, Theorem 11.6]{2007HarTer}.

In the present paper, we introduce a type of element in $\tet$ that ``looks like'' a standard generator.
For mutually distinct $h,i,j,k \in \mathbb{I}$, consider the standard generator $x_{ij}$ of $\tet$.
An element $\xi \in \tet$ is called $x_{ij}$-like whenever both (i) $\xi$ commutes with $x_{ij}$; (ii) $\xi$ and $x_{hk}$ satisfy a Dolan-Grady relation; see Definition \ref{def:xij-like}.
Pick mutually distinct $i,j,k \in \mathbb{I}$.
In our main result, we find an attractive basis for $\boxtimes$ with the property that every basis element is either $x_{ij}$-like or $x_{jk}$-like or $x_{ki}$-like.
Due to the $S_4$-symmetry on $\tet$, without loss of generality we will state our main result for the case in which $i=1$, $j=2$, $k=3$.

We now summarize our results in more detail.
For distinct $i,j \in \mathbb{I}$, let $X_{ij}$ denote the subset of $\tet$ consisting of $x_{ij}$-like elements. 
It turns out that $X_{ij}$ is a subspace of $\tet$.
We show that the $\mathbb{F}$-vector space $\tet$ is a direct sum of $X_{12}$, $X_{23}$, $X_{31}$.
Next, we present a basis for each of $X_{12}$, $X_{23}$, and $X_{31}$. 
These bases are constructed as follows.
We consider the Onsager subalgebra $\mathcal{O}$ of $\tet$ generated by $x_{12}$ and $x_{03}$.
We recursively define two sequences, $\{a_n\}_{n\geq 0}$ and $\{b_n\}_{n\geq 0}$, of elements in $\mathcal{O}$ as follows:
\begin{align}
	&& & a_0=x_{12}, && a_1=[x_{03}, a_0], && a_2=[x_{12},a_1], && a_3=[x_{03},a_2], && \ldots  &&&&\\
	&& & b_0=x_{03}, && b_1=[x_{12}, b_0], && b_2=[x_{03},b_1], && b_3=[x_{12},b_2], && \ldots &&&&
\end{align}
We will show that (i) $a_n, b_n$ are linearly independent when $n$ is even; (ii) $a_n\ne 0$ and $a_n+b_n=0$ when $n$ is odd.
We also show that the elements $a_{2n}, b_{2n}, a_{2n+1}$ $(n\geq 0)$ form a basis for $\mathcal{O}$.
We then display the Lie bracket action on this basis.
Using these basis elements, we construct a basis $\{\bx_n\}_{n \geq 0}$ (resp. $\{\by_n\}_{n \geq 0}$) (resp. $\{\bz_n\}_{n \geq 0}$) for the subspace $X_{12}\cap\mathcal{O}$ (resp. $X_{23}\cap\mathcal{O}$) (resp. $X_{31}\cap\mathcal{O}$).
Furthermore, we show that the following sums are direct:
\begin{align}
	\mathcal{O}	&= (X_{12} \cap \mathcal{O}) + (X_{23} \cap \mathcal{O}) + (X_{31} \cap \mathcal{O}), \label{intro eq(1):O DS}\\
	\mathcal{O}'	&= (X_{23} \cap \mathcal{O'}) + (X_{31} \cap \mathcal{O}') + (X_{12} \cap \mathcal{O}'), \label{intro eq(2):O DS}\\
	\mathcal{O}''	&= (X_{31} \cap \mathcal{O}'') + (X_{12} \cap \mathcal{O}'') + (X_{23} \cap \mathcal{O}''), \label{intro eq(3):O DS}
\end{align}
and
\begin{align}
	X_{12} & = (X_{12} \cap \mathcal{O}) + (X_{12} \cap \mathcal{O}^{\prime}) + (X_{12} \cap \mathcal{O}^{\prime\prime}), \label{intro eq(1):X12}\\
	X_{23} & = (X_{23} \cap \mathcal{O}) + (X_{23} \cap \mathcal{O}^{\prime}) + (X_{23} \cap \mathcal{O}^{\prime\prime}), \label{intro eq(2):X23}\\
	X_{31} & = (X_{31} \cap \mathcal{O}) + (X_{31} \cap \mathcal{O}^{\prime}) + (X_{31} \cap \mathcal{O}^{\prime\prime}), \label{intro eq(3):X31}
\end{align}
where $\mathcal{O}'$ (resp. $\mathcal{O}''$) is the Onsager subalgebra of $\tet$ generated by $x_{23}$ and $x_{01}$ (resp. $x_{31}$ and $x_{02}$).
Using these direct sums, we obtain a basis for each of $X_{12}$, $X_{23}$,  $X_{31}$.
Our findings are summarized in the following table:
\begin{equation}\label{intro:table}
{\renewcommand{\arraystretch}{1.5}
\begin{array}{c|ccccc}
\tet	& \mathcal{O} & & \mathcal{O}' & & \mathcal{O}'' \\
\hline
X_{12}	& \{\bx_{n} \}_{n\geq 0} & & \{\bz'_{n} \}_{n\geq 0} & & \{\by''_{n} \}_{n\geq 0} \\
X_{23}	& \{\by_{n} \}_{n\geq 0} & & \{\bx'_{n} \}_{n\geq 0} & & \{\bz''_{n} \}_{n\geq 0} \\
X_{31}	& \{\bz_{n} \}_{n\geq 0} & & \{\by'_{n} \}_{n\geq 0} & & \{\bx''_{n} \}_{n\geq 0} 
\end{array}}
\end{equation}
This table illustrates the following (i)--(vi):
(i) the sum of the three row-indices is direct and equal to $\tet$;
(ii) the sum of the three column-indices is direct and equal to $\tet$;
(iii) for every entry, the vectors displayed form a basis for the intersection of the row-index and the column-index;
(iv) in each row, the union of the vectors displayed forms a basis for the row-index on the left;
(v) in each column, the union of the vectors displayed forms a basis for the column-index on the top;
(vi) the union of the vectors displayed in all nine entries forms a basis for $\tet$.
We then display the Lie bracket action on the basis $\{\bx_n, \by_n, \bz_n\}_{n \geq 0}$ for $\mathcal{O}$.
We also show how the basis $\{\bx_n, \by_n, \bz_n\}_{n \geq 0}$ is recursively obtained from the standard generators of $\mathcal{O}$.
Finally, we discuss the transition matrices between the two bases $\{a_{2n}, b_{2n}, a_{2n+1}\}_{n\geq0}$ and $\{\bx_n, \by_n, \bz_n\}_{n \geq 0}$ for $\mathcal{O}$.
To obtain our main results, numerous computations are carried out in $\L$ using the Lie algebra isomorphism $\sigma: \tet \to \L$ displayed in Lemma \ref{def: sigma} (cf. \cite[Proposition 6.5]{2007HarTer}).

\smallskip
This paper is organized as follows. 
In Section \ref{pre:L(sl2)+}, we review the three-point $\mathfrak{sl}_2$ loop algebra $\L$ and its relation to $\tet$.
In Section \ref{sec:vvf:L(sl2)+}, we express an element of $\L$ in a vector-valued form; this form is used throughout the paper for computational convenience.
In Section \ref{S:x_ij-like}, we define the concept of an $x_{ij}$-like element in $\tet$ and discuss the image of $X_{ij}$ under the Lie algebra isomorphism $\sigma:\tet \to \L$. 
In addition, we show that the sum $X_{12}+X_{23}+X_{31}$ is direct and equal to $\tet$.
In Section \ref{Sec: Subsp Ui}, we focus on the Onsager subalgebra $\mathcal{O}$ of $\tet$ generated by $x_{12}$ and $x_{03}$. 
We define subspaces $\{U_n\}_{n \geq 0}$ of $\mathcal{O}$ and then decompose $\mathcal{O}$ into a direct sum of those subspaces. 
In Section \ref{sec:L(sl2)+}, we show that $\{a_{2n}, b_{2n}, a_{2n+1}\}_{n\geq 0}$ forms a basis for $\mathcal{O}$ and present the Lie bracket action on this basis.
In Section \ref{Sec:x_ij-like O}, we define the vectors $\{\bx_n\}_{n \geq 0}, \{\by_n\}_{n \geq 0}, \{\bz_n\}_{n \geq 0}$ in $\mathcal{O}$ and show that they form bases for $X_{12}\cap \mathcal{O}$, $X_{23}\cap \mathcal{O}$, $X_{31}\cap \mathcal{O}$, respectively.
We show that the sums \eqref{intro eq(1):O DS}--\eqref{intro eq(3):X31} are direct. 
We then show that $\{\bx_n, \by_n, \bz_n\}_{n \geq 0}$ forms a basis for $\tet$.
In Section \ref{sec:2nd basis O}, we display the Lie bracket action on the basis $\{\bx_n, \by_n, \bz_n\}_{n \geq 0}$ for $\mathcal{O}$.
We also discuss how the basis $\{\bx_n, \by_n, \bz_n\}_{n \geq 0}$ is recursively obtained from the standard generators of $\mathcal{O}$.
In Section \ref{sec:TM}, we discuss the transition matrices between the two bases for $\mathcal{O}$.
This paper concludes with Appendices A and B, which provide some details about the Onsager subalgebras $\mathcal{O}'$ and $\mathcal{O}''$ of $\tet$.

\begin{notation}
Throughout this paper, we denote the set of natural numbers by $\mathbb{N} = \{1, 2, 3, \ldots\}$ and the set of non-negative integers by $\mathbb{N}_0 = \{0, 1, 2, 3, \ldots \}$. 
\end{notation}

\section{The three-point $\mathfrak{sl}_2$ loop algebra $\L$}\label{pre:L(sl2)+}

In this section, we recall the Lie algebra $\mathfrak{sl}_2=\mathfrak{sl}_2(\mathbb{F})$ and the three-point $\mathfrak{sl}_2$ loop algebra $\L$. 
We then revisit how the algebra $\tet$ is related to $\mathfrak{sl}_2$ and $\L$, respectively.
We begin by recalling the Lie algebra $\mathfrak{sl}_2$, which has a basis $e, f, h$ and Lie bracket
\begin{equation*}
	[h,e] = 2e, \qquad \qquad
	[h,f] = -2f, \qquad \qquad 
	[e,f] = h.
\end{equation*}
Define
\begin{equation*}
	x=2e-h, \qquad \qquad
	y=-2f-h, \qquad \qquad 
	z=h.
\end{equation*}
By \cite[Lemma 3.2]{2007HarTer}, $\sl2$ has a basis $x,y,z$ such that
\begin{equation}\label{eq:[x,y]}
	[x,y] = 2x+2y, \qquad \qquad
	[y,z] = 2y+2z, \qquad \qquad 
	[z,x] = 2z+2x.
\end{equation}
We call $x,y,z$ the \emph{equitable basis} for $\sl2$.
By \cite[Lemma 3.4]{2007HarTer}, there exists an automorphism $\prime$ of $\sl2$ such that 
\begin{equation}\label{auto_sl2}
	x' = y, \qquad y'=z, \qquad z' = x.
\end{equation}
Note that the automorphism $\prime$ has order 3.
The Lie algebra $\sl2$ is related to $\tet$ as follows.
Let $h,i,j$ denote mutually distinct elements of $\mathbb{I}$.
By \cite[Proposition 3.6, Corollary 12.1]{2007HarTer} there exists a unique Lie algebra injective homomorphism $\phi:\sl2 \to \tet$ that sends
\begin{equation}
	x  \ \longmapsto \  x_{hi}, \qquad \qquad 
	y  \ \longmapsto \  x_{ij}, \qquad \qquad 
	z  \ \longmapsto \  x_{jh}.
\end{equation}

\medskip
Let $t$ be an indeterminate, and define 
\begin{equation*}
	\mathcal{A} = \mathbb{F}[t, t^{-1}, (t-1)^{-1}],
\end{equation*}
as the $\mathbb{F}$-algebra of all Laurent polynomials in $t$ and $(t-1)^{-1}$ with coefficients in $\mathbb{F}$.
By \cite[Lemma 6.2]{2007HarTer}, there exists a unique $\mathbb{F}$-algebra automorphism $\prime$ of $\mathcal{A}$ that sends $t$ to $1-t^{-1}$. 
Observe that
\begin{equation}\label{auto_A}
	t'=1-t^{-1}, \qquad \qquad  t'' = (1-t)^{-1}, \qquad \qquad t'''=t.
\end{equation}

\begin{lemma}\label{lem:basis for A} 
The following {\rm(i)} and {\rm(ii)} hold.
\begin{enumerate}[label=(\roman*), font=\normalfont]
	\item The $\mathbb{F}$-vector space $\mathcal{A}$ has a basis
	\begin{equation}
		\{1\} \cup \{ t^i, (t')^i, (t'')^i \mid i \in \mathbb{N}\}.
	\end{equation}
	\item The $\mathbb{F}$-vector space $\mathcal{A}$ satisfies
	\begin{alignat}{6}
	\mathcal{A}	
		& \ = \ && \mathbb{F}[t] & \ + \ & (1-t')  &&\mathbb{F}[t'] & \ + \ & t'' && \mathbb{F}[t''] \label{ds:A(1)}\\
		& \ = \ && \mathbb{F}[t'] & \ + \ & (1-t'') &&\mathbb{F}[t''] &\ + \ & \ t && \mathbb{F}[t] \label{ds:A(2)}\\
		& \ = \ && \mathbb{F}[t''] & \ + \ & \ (1-t)  &&\mathbb{F}[t] & \ + \ &t' && \mathbb{F}[t']. \label{ds:A(3)}
	\end{alignat}
	Each of the sums \eqref{ds:A(1)}--\eqref{ds:A(3)} is direct.
\end{enumerate}
\end{lemma}
\begin{proof}
By construction.
\end{proof}

We now recall the three-point $\sl2$ loop algebra.
\begin{definition}[{\cite[Definition 6.1]{2007HarTer}}] \label{Def:3pt sl2-loop alg}
Let $L(\sl2)^+$ denote the Lie algebra over $\mathbb{F}$ consisting of the $\mathbb{F}$-vector space $\sl2 \otimes \mathcal{A}$, where $\otimes=\otimes_{\mathbb{F}}$, and Lie bracket
\begin{equation}\label{def:eq[u,v]x(ab)}
	[u\otimes a, v \otimes b] = [u,v] \otimes ab, \qquad \qquad u,v \in \sl2, \qquad a, b \in \mathcal{A}. 
\end{equation}
We call $L(\sl2)^+$ the \emph{three-point $\sl2$ loop algebra}.
\end{definition}

\noindent
We note that the algebra $\L$ is a right $\mathcal{A}$-module with the action map: 
\begin{align*}
	\L \times \mathcal{A} \ \longrightarrow \ \L, \qquad 
	(u\otimes a , b) \ \longmapsto \ u\otimes ab.
\end{align*}

\begin{lemma}\label{lem:basis for L(sl2)+} 
The following {\rm(i)} and {\rm(ii)} hold.
\begin{enumerate}[label=(\roman*), font=\normalfont]
	\item The $\mathbb{F}$-vector space $L(\sl2)^+$ has a basis
	\begin{align*}
		\{ x\otimes1, y\otimes1, z\otimes 1 \} 
		&\cup \{ x\otimes t^i, y\otimes t^i, z\otimes t^i \mid i \in \mathbb{N} \} \\
		&\cup \{ x\otimes (t')^i, y\otimes (t')^i, z\otimes (t')^i \mid i \in \mathbb{N} \} \\
		&\cup \{ x\otimes (t'')^i, y\otimes (t'')^i, z\otimes (t'')^i \mid i \in \mathbb{N} \}.
	\end{align*}

	\item The $\mathbb{F}$-vector space $L(\sl2)^+$ satisfies
	\begin{equation}\label{ds:L(sl2)+}
	\begin{array}{llllllll}
	\L	& = & & x\otimes \mathbb{F}[t] & + &x\otimes (1-t')\mathbb{F}[t'] & + & x\otimes t''\mathbb{F}[t'']  \\
		&  & + & y\otimes \mathbb{F}[t']  & + & y\otimes (1-t'')\mathbb{F}[t''] & + & y\otimes t\mathbb{F}[t] \\
		&  & + & z\otimes \mathbb{F}[t'']  & + & z\otimes (1-t)\mathbb{F}[t] & + & z\otimes t'\mathbb{F}[t']. 
	\end{array}
	\end{equation}
	The sum \eqref{ds:L(sl2)+} is direct.
\end{enumerate}
\end{lemma}
\begin{proof}
By construction.
\end{proof}

We recall how the algebra $\tet$ is related to $\L$.
\begin{lemma}[{\cite[Proposition 6.5]{2007HarTer}}]\label{def: sigma}
There exists a unique Lie algebra isomorphism $\sigma: \tet \to L(\sl2)^+$ that sends
\begin{align}
	&& x_{12} \quad & \longmapsto \quad x \otimes 1, && x_{03} \quad \longmapsto \quad  y\otimes t + z \otimes (t-1), &&\\
	&& x_{23} \quad & \longmapsto \quad y \otimes 1, && x_{01} \quad \longmapsto \quad  z\otimes t' + x \otimes (t'-1), &&\\
	&& x_{31} \quad & \longmapsto \quad z \otimes 1, && x_{02} \quad \longmapsto \quad  x\otimes t'' + y \otimes (t''-1), &&
\end{align}
where $x,y,z$ is the equitable basis for $\sl2$.
\end{lemma}

\begin{definition}\label{def:sg L}
Recall the standard generators of $\tet$ from \eqref{standard generators tet}.
Let $\sigma$ be the Lie algebra isomorphism from $\tet \to L(\sl2)^+$ as in Lemma \ref{def: sigma}.
By the \emph{standard generator of $\L$}, we mean the $\sigma$-image of the standard generator for $\tet$.
\end{definition}

\begin{lemma}\label{lem:[x_ij,x_hk]^sigma}
The map $\sigma$ sends
\begin{align}
	[x_{12}, x_{03}]\quad & \longmapsto \quad 2x \otimes 1 + 2y \otimes t + 2z \otimes (1-t), \\
	[x_{23}, x_{01}]\quad & \longmapsto \quad 2y \otimes 1 + 2z \otimes t' + 2x \otimes (1-t'), \\
	[x_{31}, x_{02}]\quad & \longmapsto \quad 2z \otimes 1 + 2x \otimes t'' + 2y \otimes (1-t'').
\end{align}
\end{lemma}
\begin{proof}
Use Lemma \ref{def: sigma}.
\end{proof}

\begin{lemma}[{\cite[Lemma 6.7]{2007HarTer}}]\label{lem:auto L(sl2)+}
There exists an automorphism $\prime$ of $L(\sl2)^+$ that satisfies
\begin{equation}
	(u \otimes a)' = u' \otimes a', \qquad \qquad u\in \sl2, \qquad a\in \mathcal{A},
\end{equation}
where $u'$ is from \eqref{auto_sl2} and $a'$ is from \eqref{auto_A}.
This automorphism has order three. 
Moreover, the following diagram commutes:
\begin{equation}\label{com diag:tet L(sl2)+}
	\begin{tikzcd}
	\tet \arrow{r}{\sigma} \arrow[swap]{d}{(123)} & L(\sl2)^+ \arrow{d}{\prime} \\%
	\tet \arrow{r}{\sigma}& L(\sl2)^+
	\end{tikzcd}
\end{equation}
\end{lemma}

\section{Vector-valued form in $L(\mathfrak{sl}_2)^+$}\label{sec:vvf:L(sl2)+}

Recall the three-point $\mathfrak{sl}_2$ loop algebra $L(\mathfrak{sl}_2)^+$.
For an element $x\otimes f + y\otimes g + z\otimes h$ in $L(\mathfrak{sl}_2)^+$, we represent it in vector-valued form as follows:
\begin{equation}\label{eq:vvf}
	x\otimes f + y\otimes g + z\otimes h 
	\qquad \longleftrightarrow \qquad
	\begin{pmatrix} f \\ g \\ h \end{pmatrix}.
\end{equation}
For the rest of this paper, we often identify an element in $L(\mathfrak{sl}_2)^+$ with its corresponding vector-valued form as shown in \eqref{eq:vvf}.
For example, the element $x\otimes (t-1) + y\otimes t^{-1} + z\otimes (t^2+1)$ is the same expression as $\begin{pmatrix} t-1 \\ t^{-1} \\ t^2+1\end{pmatrix}$.
Recall the automorphism $\prime$ of $L(\sl2)^+$ from Lemma \ref{lem:auto L(sl2)+}.
Then we have
\begin{equation}
	\begin{pmatrix} f \\ g \\ h \end{pmatrix}'
	= \begin{pmatrix} h' \\ f' \\ g' \end{pmatrix}.
\end{equation}
The following lemma is useful for computing commutators of two vectors in $L(\sl2)^+$.
\begin{lemma}\label{v-val comp}
For $f, g,h  \in \mathcal{A}$, the following {\rm(i)}--{\rm(iii)} hold.
\begin{enumerate}[label=(\roman*), font=\normalfont]
	\item $\left[\begin{pmatrix} f \\ 0 \\ 0 \end{pmatrix}, \begin{pmatrix} 0 \\ g \\ 0 \end{pmatrix}\right] = 2\begin{pmatrix} fg \\ fg \\ 0 \end{pmatrix}.$
	\item $\left[\begin{pmatrix} 0 \\ g \\ 0 \end{pmatrix}, \begin{pmatrix} 0 \\ 0 \\ h \end{pmatrix}\right] = 2\begin{pmatrix} 0 \\ gh \\ gh  \end{pmatrix}.$
	\item $\left[\begin{pmatrix} 0 \\ 0 \\ h  \end{pmatrix}, \begin{pmatrix} f \\ 0 \\ 0 \end{pmatrix}\right] = 2\begin{pmatrix} fh \\ 0 \\ fh \end{pmatrix}.$
\end{enumerate}
\end{lemma}
\begin{proof}
(i): We evaluate $[x\otimes f, y\otimes g]$. 
By using \eqref{eq:[x,y]} and \eqref{def:eq[u,v]x(ab)},
$$
	[x\otimes f, y\otimes g] = [x,y] \otimes fg = (2x+2y)\otimes fg = 2x\otimes fg + 2y\otimes fg.
$$
By applying the identification \eqref{eq:vvf}, we obtain (i).\\
(ii), (iii): Similar to (i).
\end{proof}

\section{The $x_{ij}$-like elements in $\tet$}\label{S:x_ij-like}

Recall the Lie algebra isomorphism $\sigma: \tet \to L(\sl2)^+$ in Lemma \ref{def: sigma} and the standard generator $x_{ij}$ of $\tet$.
In this section, we define the notion of an $x_{ij}$-like element in $\tet$.
As we discussed in Section \ref{Intro}, our goal is to classify the $x_{ij}$-like elements in $\tet$.
To do this, we will calculate $x^\sigma_{ij}$ in $\L$ instead of directly working in $\tet$, as this approach simplifies the computation.

\begin{definition}\label{def:xij-like}
Recall the Lie algebra $\tet$ and its standard generators $x_{ij}$ from Definition \ref{Def:Tet-alg}.
By an \emph{$x_{ij}$-like element in $\tet$}, we mean an element $\xi$ in $\tet$ that satisfies the following conditions:
\begin{align}
	[x_{ij}, \xi] & = 0,\label{x_ij like(1)}  \\
	[x_{hk}, [x_{hk}, [x_{hk}, \xi]]] & = 4[x_{hk}, \xi], \label{x_ij like(2)}
\end{align}
where $h,i,j,k$ are mutually distinct.
\end{definition}

\begin{definition}\label{del:X_ij}
For each standard generator $x_{ij}$ of $\tet$, let $X_{ij}$ denote the subset of $\tet$ consisting of the $x_{ij}$-like elements in $\tet$.
We note that $X_{ij}$ is a subspace of $\tet$.
\end{definition}
We will find the elements of $X_{ij}$.
To do this, we consider the $\sigma$-images of the elements in $X_{ij}$.

\begin{definition}\label{Def:x_ij L(sl2)+}
For an element $\theta \in L(\sl2)^+$, we say $\theta$ is an \emph{$x^\sigma_{ij}$-like element in $\L$} whenever $\theta$ is the $\sigma$-image of an $x_{ij}$-like element in $\tet$.
That is, for an element $\xi \in \tet$, $\xi$ is an $x_{ij}$-like element in $\tet$ if and only if $\xi^\sigma$ is an $x^\sigma_{ij}$-like element in $\L$.
\end{definition}

\begin{lemma}\label{lem:x sigma-like}
For a standard generator $x^\sigma_{ij}$ of $\L$, 
an element $\theta\in \L$ is an $x^\sigma_{ij}$-like element in $\L$ if and only if both
\begin{align}
	[x^\sigma_{ij}, \theta] & = 0,\label{x^s_ij like(1)}  \\
	[x^\sigma_{hk}, [x^\sigma_{hk}, [x^\sigma_{hk}, \theta]]] & = 4[x^\sigma_{hk}, \theta], \label{x^s_ij like(2)}
\end{align}
where $h,i,j,k$ are mutually distinct.
\end{lemma}
\begin{proof}
Immediate from Definitions \ref{def:xij-like} and \ref{Def:x_ij L(sl2)+}.
\end{proof}

\noindent
Consider the $\sigma$-image of $X_{ij}$.
Observe that $X^\sigma_{ij}$ is the subspace of $L(\sl2)^+$ consisting of the $x^\sigma_{ij}$-like elements in $L(\sl2)^+$.

\begin{lemma}\label{lem:x_ij-elts L(sl2)+}
Recall $\mathcal{A}=\mathbb{F}[t, t^{-1}, (t-1)^{-1}]$.
We have
\begin{align}
	&& X^\sigma_{12} & = x \otimes \mathcal{A}, && X^\sigma_{03} = (y\otimes t + z \otimes (t-1))\mathcal{A}, && \label{eq(1):X12, X03 sigma}\\
	&& X^\sigma_{23} & = y \otimes \mathcal{A}, && X^\sigma_{01} = (z\otimes t' + x \otimes (t'-1))\mathcal{A}, && \label{eq(2):X23, X01 sigma}\\
	&& X^\sigma_{31} & = z \otimes \mathcal{A}, && X^\sigma_{02} = (x\otimes t'' + y \otimes (t''-1))\mathcal{A}. &&\label{eq(3):X31, X02 sigma}
\end{align}
\end{lemma}

\begin{proof}
First we show that $X^\sigma_{12} = x \otimes \mathcal{A}$.
Recall the standard generators $x^\sigma_{12} = x\otimes1$ and $x^\sigma_{03} = y\otimes t + z \otimes (t-1)$.
For any $x\otimes f \in x\otimes \mathcal{A}$, we routinely check that both 
\begin{align*}
	[x^\sigma_{12}, x\otimes f]  & = 0,\\
	[x^\sigma_{03}, [x^\sigma_{03}, [x^\sigma_{03}, x\otimes f ]]] & = 4[x^\sigma_{03}, x\otimes f ].
\end{align*}
By Lemma \ref{lem:x sigma-like}, the element $x\otimes f$ is an $x^\sigma_{12}$-like element in $\L$.
Therefore, $x\otimes f \in X_{12}^\sigma$.
Conversely, let $\theta \in X_{12}^\sigma$.
Write $\theta = x\otimes f + y\otimes g + z\otimes h$.
By \eqref{x^s_ij like(1)} at $(i,j)=(1,2)$, we have
\begin{align*}
	0 = [x_{12}^\sigma, \theta] & = [x\otimes 1,  x\otimes f + y\otimes g + z\otimes h] \\
	& = 2x\otimes(g-h) + 2y\otimes g - 2z\otimes h.
\end{align*}
From this equation, we find $g=0$ and $h=0$, implying $\theta=x\otimes f$.
Therefore, $\theta \in x\otimes \mathcal{A}$. 
The desired result follows.
To obtain the left equations in \eqref{eq(2):X23, X01 sigma} and \eqref{eq(3):X31, X02 sigma}, apply the automorphism $\prime$ of Lemma \ref{lem:auto L(sl2)+} to $X^\sigma_{12}=x\otimes \mathcal{A}$.

Next, we show $X^\sigma_{03} = (y\otimes t + z \otimes (t-1))\mathcal{A}$.
Consider an element $y\otimes tf +z\otimes (t-1)f$ for $f\in \mathcal{A}$.
Then we routinely check that both
\begin{align*}
	[x^\sigma_{03}, y\otimes tf +z\otimes (t-1)f]  & = 0,\\
	[x^\sigma_{12}, [x^\sigma_{12}, [x^\sigma_{12}, y\otimes tf +z\otimes (t-1)f ]]]  &= 4[x^\sigma_{12}, y\otimes tf +z\otimes (t-1)f ].
\end{align*}
By Lemma \ref{lem:x sigma-like}, $y\otimes tf +z\otimes (t-1)f$ is an $x_{03}^\sigma$-like element in $\L$.
Therefore, $y\otimes tf +z\otimes (t-1)f \in X_{03}^\sigma$.
Conversely, let $\omega \in X_{03}^\sigma$.
Write $\omega = x\otimes f + y\otimes g + z\otimes h$.
By \eqref{x^s_ij like(1)} at $(i,j)=(0,3)$, we have
\begin{align*}
	0 = [x_{03}^\sigma, \omega] 
	& = [y\otimes t + z \otimes (t-1),  x\otimes f + y\otimes g + z\otimes h] \\
	& = -2 x\otimes f + 2y\otimes \big(2tf + 2(1-t)g + 2th \big)
	 + 2z\otimes \big(2(t-1)f + 2(1-t)g + 2th \big).
\end{align*}
From this equation, we find $f=0$ and $th=(t-1)g$.
Using these equations, we have
\begin{equation}
	\omega = y\otimes g + z\otimes h = (y\otimes tg + z\otimes th) t^{-1}  = (y\otimes t + z\otimes (t-1)) gt^{-1}.
\end{equation}
Since $gt^{-1}\in \mathcal{A}$, it follows $\omega \in (y\otimes t + z \otimes (t-1))\mathcal{A}$.
The desired result follows.
To obtain the right equations in \eqref{eq(2):X23, X01 sigma} and \eqref{eq(3):X31, X02 sigma}, apply the automorphism $\prime$ of Lemma \ref{lem:auto L(sl2)+} to $X^\sigma_{03} = (y\otimes t + z \otimes (t-1))\mathcal{A}$.
\end{proof}

\begin{corollary}
For each standard generator $x_{ij}^\sigma$ of $\L$, we have
\begin{equation}
	X^\sigma_{ij} = x_{ij}^\sigma \mathcal{A}.
\end{equation}
\end{corollary}
\begin{proof}
Since $\L$ is a right $\mathcal{A}$-module and by Lemmas \ref{def: sigma} and \ref{lem:x_ij-elts L(sl2)+}, the result follows.
\end{proof}

Next, we discuss a basis for each subspace $X_{12}^\sigma$, $X_{23}^\sigma$, and $X_{31}^\sigma$ of $\L$.
\begin{lemma}\label{lem:L_ij like} 
The following {\rm(i)}--{\rm(iii)} hold.
\begin{enumerate}[label=(\roman*), font=\normalfont]
	\item The subspace $X^\sigma_{12}$ has a basis
	\begin{equation}\label{basis X12}
		\{ x\otimes1 \} \cup \{ x\otimes t^i, x\otimes (t')^i, x\otimes (t'')^i \mid i \in \mathbb{N}\}.
	\end{equation}
	\item The subspace $X^\sigma_{23}$ has a basis
	\begin{equation}\label{basis X23}
		\{ y\otimes1 \} \cup \{ y\otimes t^i, y\otimes (t')^i, y\otimes (t'')^i \mid i \in \mathbb{N}\}.
	\end{equation}
	\item	The subspace $X^\sigma_{31}$ has a basis
	\begin{equation}\label{basis X31}
		\{ z\otimes 1 \} \cup \{ z\otimes t^i, z\otimes (t')^i, z\otimes (t'')^i \mid i \in \mathbb{N}\}.
	\end{equation}
\end{enumerate}
Moreover, the $\mathbb{F}$-vector space $L(\sl2)^+$ satisfies
\begin{equation}\label{eq:L=ds(Xij)}
	\L = X^\sigma_{12} + X^\sigma_{23} + X^\sigma_{31} \qquad (\text{\rm direct sum}).
\end{equation} 
\end{lemma}
\begin{proof}
By Lemma \ref{lem:basis for L(sl2)+}(i) and Lemma \ref{lem:x_ij-elts L(sl2)+}.
\end{proof}

\begin{corollary}
We have
\begin{alignat}{5}
	X^\sigma_{12} & \ = \ && x \otimes \mathbb{F}[t] & \ + \ & x \otimes (1-t')\mathbb{F}[t'] & \ + \ & x \otimes t''\mathbb{F}[t''], \label{eq:X12-DS}\\
	X^\sigma_{23} & \ = \ && y \otimes \mathbb{F}[t'] & \ + \ & y \otimes (1-t'')\mathbb{F}[t''] & \ + \ & y \otimes t\mathbb{F}[t], \label{eq:X23-DS}\\
	X^\sigma_{31} & \ = \ && z \otimes \mathbb{F}[t''] & \ + \ & z \otimes (1-t)\mathbb{F}[t] & \ + \ & z \otimes t'\mathbb{F}[t']. \label{eq:X31-DS}
\end{alignat}
Each of the sums \eqref{eq:X12-DS}--\eqref{eq:X31-DS} is direct.
\end{corollary}
\begin{proof}
The sum on the right-hand side of \eqref{eq:X12-DS} has a basis \eqref{basis X12}.
Therefore, we have \eqref{eq:X12-DS}.
Similarly, we have \eqref{eq:X23-DS} and \eqref{eq:X31-DS}.
\end{proof}

\begin{corollary}
We have 
\begin{equation}\label{cor:tet=ds(Xij)}
	\tet = X_{12} + X_{23} + X_{31} \qquad (\text{\rm direct sum}).
\end{equation}
\end{corollary}
\begin{proof}
By \eqref{eq:L=ds(Xij)}.
\end{proof}


\section{Subspaces $\{U_i\}_{i\in \mathbb{N}_0}$ of $\mathcal{O}$}\label{Sec: Subsp Ui}

For the rest of this paper, we denote by $\mathcal{O}$ the Lie subalgebra of $\tet$ generated by $x_{12}$ and $x_{03}$.
Note that $\mathcal{O}$ is the Onsager subalgebra with the standard generators $x_{12}$ and $x_{03}$.
In this section, we introduce the subspaces $\{U_i\}_{i \in \mathbb{N}_0}$ of $\mathcal{O}$.
We will show that $\mathcal{O}$ is decomposed into a direct sum of $\{U_i\}_{i \in \mathbb{N}_0}$.
First, we define a sequence of subspaces $\{\mathcal{U}_i\}_{i\in \mathbb{N}_0}$ of $\mathcal{O}$ as follows:
\begin{equation}
	\mathcal{U}_0 = \operatorname{Span}\{x_{12}, x_{03}\}, 
	\qquad \qquad 
	\mathcal{U}_{i+1} = \mathcal{U}_{i} + [\mathcal{U}_{0}, \mathcal{U}_{i}] \qquad (i \in \mathbb{N}_0).
\end{equation}
From the definition, the sequence $\{\mathcal{U}_i\}_{i\in \mathbb{N}_0}$ forms a nested sequence of subspaces.
\begin{lemma}\label{lem:O=unionU}
We have
\begin{equation}\label{eq:O=uni U}
	\mathcal{O} = \bigcup_{i\in \mathbb{N}_0}\mathcal{U}_i.
\end{equation}
\end{lemma}
\begin{proof}
Let $\widetilde{\mathcal{O}}$ denote the right-hand side of \eqref{eq:O=uni U}.
We prove $\mathcal{O}=\widetilde{\mathcal{O}}$.
To this end, we assert that $\widetilde{\mathcal{O}}$ is an ideal of $\mathcal{O}$.
Define the subspace $\Omega$ of $\mathcal{O}$ as follows:
\begin{equation*}
	\Omega = \left\{ \omega \in \mathcal{O} : [w, \widetilde{\mathcal{O}}] \subseteq \widetilde{\mathcal{O}} \right\}.
\end{equation*}
By construction, we have $[\Omega,  \widetilde{\mathcal{O}}] \subseteq \widetilde{\mathcal{O}}$.
To show our assertion, it suffices to show $\Omega = \mathcal{O}$.
We observe that $x_{12}, x_{03}\in \Omega$. 
Now, consider any two elements $u$ and $v$ in $\Omega$.
By the Jacobi identity, the following hold for all $z\in \widetilde{\mathcal{O}}$:
\begin{equation}\label{eq:JI in O}
	[u,[v,z]] +  [v,[z,u]] + [z,[u,v]]=0.
\end{equation}
By construction, the first and second terms on the left-hand side of \eqref{eq:JI in O} belong to $\widetilde{\mathcal{O}}$.
Therefore, the third term also belongs to $\widetilde{\mathcal{O}}$.
Consequently we have $[u,v] \in \Omega$, implying that $\Omega$ is closed under the Lie bracket.
Therefore, $\Omega$ is a subalgebra of $\mathcal{O}$.
Since $\Omega$ contains the generators $x_{12}$, $x_{03}$ of $\mathcal{O}$, we conclude that $\Omega=\mathcal{O}$.
We have shown that $\widetilde{\mathcal{O}}$ is an ideal of $\mathcal{O}$. 
Note that $x_{12}, x_{03} \in \mathcal{U}_0 \subset \widetilde{\mathcal{O}}$.
Since $\widetilde{\mathcal{O}}$ is an ideal containing the generators $x_{12}$, $x_{03}$ of $\mathcal{O}$, we have $\mathcal{O}=\widetilde{\mathcal{O}}$.
The result follows.
\end{proof}

\begin{definition}\label{def:seq a_i b_i}
We define the sequences $\{a_i\}_{i \in \mathbb{N}_0}$ and $\{b_i\}_{i \in \mathbb{N}_0}$ of elements in $\mathcal{O}$ by 
\begin{align}
	&& a_0 & = x_{12},	& a_{2j-1} &= [b_0, a_{2j-2}], & a_{2j} &= [a_0, a_{2j-1}] &&(j=1,2,3\ldots), \label{Def:a_i}\\
	&& b_0 &= x_{03}, & b_{2j-1} &= [a_0, b_{2j-2}], & b_{2j} &= [b_0, b_{2j-1}] && (j=1,2,3\ldots).\label{Def:b_i}
\end{align}
\end{definition}
\noindent
For example, the elements $a_i$ and $b_i$ are described in the following table:
\begin{equation}\label{table:U_i,a,b}
{\renewcommand{\arraystretch}{1.5}
\begin{array}{cccc}
\quad i \quad \ & a_i & & b_i \\
\hline
0 & x_{12} & &x_{03} \\
1 & [x_{03}, x_{12}] & & [x_{12}, x_{03}] \\
2 & [x_{12}, [x_{03}, x_{12}]] & & [x_{03}, [x_{12}, x_{03}]] \\
3 & [x_{03}, [x_{12}, [x_{03}, x_{12}]]] & & [x_{12},[x_{03}, [x_{12}, x_{03}]]] \\
4 & [x_{12}, [x_{03}, [x_{12}, [x_{03}, x_{12}]]]] & & [x_{03},[x_{12},[x_{03}, [x_{12}, x_{03}]]]] \\
5 & [x_{03}, [x_{12}, [x_{03}, [x_{12}, [x_{03}, x_{12}]]]]] & & [x_{12},[x_{03},[x_{12},[x_{03}, [x_{12}, x_{03}]]]]] \\
\vdots & \vdots & & \vdots
\end{array}
}
\end{equation}
\begin{lemma}\label{lem:cal(U)_i}
For each $i \in \mathbb{N}_0$, the elements $a_i$ and $b_i$ belong to $\mathcal{U}_i$.
\end{lemma}
\begin{proof}
It is clear that $a_0, b_0 \in \mathcal{U}_0$. 
For each integer $i\geq 1$, using \eqref{Def:a_i} and \eqref{Def:b_i}, we have both
$$
	a_i \in [\mathcal{U}_0, a_{i-1}] \subset [\mathcal{U}_0, \mathcal{U}_{i-1}], \qquad \qquad
	b_i \in [\mathcal{U}_0, b_{i-1}] \subset [\mathcal{U}_0, \mathcal{U}_{i-1}].
$$
Since $[\mathcal{U}_0, \mathcal{U}_{i-1}] \subset \mathcal{U}_i$, the result follows.
\end{proof}

\noindent
By Lemma \ref{lem:cal(U)_i} and since the sequence $\{\mathcal{U}_i\}_{i\in \mathbb{N}_0}$ is nested, the set $\{a_j, b_j \mid 0\leq j \leq i\}$ is contained in $\mathcal{U}_i$.
For notational convention, we define the linear maps $G, H : \tet \to \tet$ by
\begin{align}
	G: \quad z \quad \longmapsto \quad [x_{12}, [x_{03}, z ]], \label{eq:op G}\\
	H: \quad z \quad \longmapsto \quad [x_{03}, [x_{12}, z ]],\label{eq:op H}
\end{align}
for all $z \in \tet$.
Note that the subalgebra $\mathcal{O}$ is invariant under both $G$ and $H$.
Using $G$, $H$ the elements $a_i$, $b_i$ are described as follows: for $i\in \mathbb{N}_0$
\begin{align}
	&& a_{2i}  &= G^i(x_{12}), & b_{2i} &= H^i(x_{03}), && \label{eq:a_i, b_i(1)}\\
	&& a_{2i+1}  &= [x_{03},G^i(x_{12})], & b_{2i+1} &= [x_{12}, H^i(x_{03})].&& \label{eq:a_i, b_i(2)} 
\end{align}

\begin{definition}\label{def:U_i}
For $i\in \mathbb{N}_0$, define the subspace $U_i$ by 
\begin{equation}
	U_i=\operatorname{Span}\{a_i, b_i\}.
\end{equation}
\end{definition}

\noindent
By \eqref{eq:a_i, b_i(1)}, \eqref{eq:a_i, b_i(2)} the following hold for all $i\in \mathbb{N}_0$:
\begin{equation}\label{eq:U_2i, U_2i+1}
	U_{2i} = \operatorname{Span}\{ G^i(x_{12}), H^i(x_{03})\},  
	\qquad 
	U_{2i+1} = \operatorname{Span}\{ [x_{03}, G^i(x_{12})], [x_{12}, H^i(x_{03})] \}.
\end{equation}
\begin{example}
For $0\leq i \leq 8$, we display $a_i$, $b_i$, and their span:
\begin{equation}\label{table:U_i,a,b G,H}
{\renewcommand{\arraystretch}{1.5}
\begin{array}{cccc}
\quad i \quad \ & a_i & b_i & \text{Span}\{a_i, b_i\}\\
\hline
0 & x_{12} & x_{03} & U_0\\
1 & [x_{03}, x_{12}] & [x_{12}, x_{03}] & U_1\\
2 & G(x_{12}) & H(x_{03}) & U_2\\
3 & [x_{03}, G(x_{12})] & [x_{12},H(x_{03})] & U_3\\
4 & G^2(x_{12}) & H^2(x_{03}) & U_4\\
5 & [x_{03}, G^2(x_{12})] & [x_{12},H^2(x_{03})] & U_5\\
6 & G^3(x_{12}) & H^3(x_{03}) & U_6\\
7 & [x_{03}, G^3(x_{12})] & [x_{12},H^3(x_{03})] & U_7\\
8 & G^4(x_{12}) & H^4(x_{03}) & U_8
\end{array}
}
\end{equation}
\end{example}

\noindent
We note that $a_i$, $b_i$ might not be a basis for $U_i$ due to the possibility of their linear dependence.

\begin{lemma}\label{lem:subspace U_i}
The following {\rm(i)}--{\rm(iv)} hold.
\begin{enumerate}[label=(\roman*), font=\normalfont]
	\item For each odd $i\in \mathbb{N}_0$ we have $a_i \ne 0$ and $b_i=-a_i$. The element $a_i$ is a basis for $U_i$.
	\item For each even $i\in \mathbb{N}_0$ the elements $a_i$, $b_i$ are linearly independent. The elements $a_i, b_i$ are a basis for $U_i$.
	\item For $i \in \mathbb{N}_0$, 
	\begin{equation}
		\mathcal{U}_i = U_0 + U_1 + \cdots + U_i \qquad (\text{\rm direct sum}).
	\end{equation}
	\item The Onsager subalgebra $\mathcal{O}$ satisfies
	\begin{equation}\label{eq(2):sumUi=O}
		\mathcal{O} = \sum_{i \in \mathbb{N}_0} U_i \qquad (\text{\rm direct sum}).
	\end{equation}
\end{enumerate}
\end{lemma}

\noindent
The proof of Lemma \ref{lem:subspace U_i} will appear in the next section.

\begin{remark}\label{rmk:O', O''}
In Section \ref{Intro}, we mentioned that the symmetric group $S_4$ acts on $\tet$.
We view each element of $S_4$ as an automorphism of $\tet$.
Let $\prime$ denote the permutation $(123)\in S_4$.
We consider the images of $\mathcal{O}$ under the maps $\prime$ and $\prime\prime$.
Then $\mathcal{O}'$ (resp. $\mathcal{O}''$) is the subalgebra of $\tet$ generated by $x_{23}$ and $x_{01}$ (resp. $x_{31}$ and $x_{02}$).
We note that $\mathcal{O}'$ and $\mathcal{O}''$ are the Onsager subalgebras of $\tet$.
Moreover, by \cite[Theorem 11.6]{2007HarTer} we have
\begin{equation}
	\tet = \mathcal{O} + \mathcal{O}' + \mathcal{O}'' \qquad \text{(direct sum)}.
\end{equation}
In Sections \ref{Sec: Subsp Ui}--\ref{Sec:x_ij-like O}, we have presented a number of results for $\mathcal{O}$. 
By applying the automorphisms $\prime$ and $\prime\prime$ to these results, we derive similar results for $\mathcal{O}'$ and $\mathcal{O}''$. 
Refer to the Appendix for the detailed data.
\end{remark}

\section{A basis for $\mathcal{O}$}\label{sec:L(sl2)+}

In the previous section, we discussed the elements $a_i$, $b_i$ $(i\in \mathbb{N}_0)$ of $\mathcal{O}$ along with the subspaces $\{U_i\}_{i\in \mathbb{N}_0}$ of $\mathcal{O}$.
In this section, we will show that the elements $a_{2i}, b_{2i}, a_{2i+1}$ $(i\in \mathbb{N}_0)$ form a basis for $\mathcal{O}$.
To do this, we will apply the Lie algebra isomorphism $\sigma: \tet \to L(\sl2)^+$ from Lemma \ref{def: sigma} to $a_i$ and $b_i$. 
We will show that the image $a^\sigma_{2i}$, $b^\sigma_{2i}$, $a^\sigma_{2i+1}$ $(i\in \mathbb{N}_0)$ form a basis for $\mathcal{O}^\sigma$.

\medskip
Recall the linear operators $G$, $H$ from \eqref{eq:op G}, \eqref{eq:op H}.
Consider the compositions:
\begin{align}
	\sigma G \sigma^{-1} : \quad & L(\sl2)^+ \  \xrightarrow{\quad \sigma^{-1} \quad } \ 
	\tet \  \xrightarrow{\quad G \quad } \ 
	\tet \  \xrightarrow{\quad \sigma \quad } \
	L(\sl2)^+,\\
	\sigma H \sigma^{-1} : \quad & L(\sl2)^+ \  \xrightarrow{\quad \sigma^{-1} \quad } \ 
	\tet \  \xrightarrow{\quad H \quad } \ 
	\tet \  \xrightarrow{\quad \sigma \quad } \
	L(\sl2)^+.
\end{align}
The maps $\sigma G \sigma^{-1}$ and $\sigma H \sigma^{-1}$ are described as follows.
For $u \in \sl2$ and $\varphi \in \mathcal{A}$, the map $\sigma G \sigma^{-1}$ sends
\begin{equation}\label{eq:action G}
	u \otimes \varphi \quad \longmapsto \quad [x \otimes 1, [ y\otimes t + z\otimes (t-1), u\otimes \varphi]], 
\end{equation}
and $\sigma H \sigma^{-1}$ sends
\begin{equation}\label{eq:action H}
	u\otimes \varphi \quad \longmapsto \quad [y\otimes t + z\otimes (t-1), [x \otimes 1, u\otimes \varphi]]. 
\end{equation}
Observe that $\sigma G \sigma^{-1} (x_{ij}^\sigma) = (G(x_{ij}))^\sigma$ and $\sigma H \sigma^{-1} (x_{ij}^\sigma) = (H(x_{ij}))^\sigma$.
The vector-valued form \eqref{eq:vvf} in $L(\sl2)^+$ is useful when we compute the actions of $\sigma G \sigma^{-1}$ and $\sigma H \sigma^{-1}$ on $L(\sl2)^+$.
This is illustrated in the following lemma.

\begin{lemma}\label{lem:operator G}
For  $f,g,h \in \mathcal{A}$,
\begin{align}
	\sigma G \sigma^{-1}: \quad 
	\begin{pmatrix}
	f \\ g\\ h
	\end{pmatrix}
	& \quad \longmapsto \quad (-4) 
	\begin{pmatrix}
	(2t-1)f \\ tf+(t-1)g-th \\ (t-1)f + (1-t)g + th
	\end{pmatrix}, \label{eq: op G}\\
	\sigma H \sigma^{-1}: \quad 
	\begin{pmatrix}
	f \\ g\\ h
	\end{pmatrix}
	& \quad \longmapsto \quad (-4) 
	\begin{pmatrix}
	g-h \\ (2t-1)g \\ (2t-1)h
	\end{pmatrix}. \label{eq: op H}
\end{align}
\end{lemma}
\begin{proof}
By \eqref{eq:action G}, \eqref{eq:action H}, 
\begin{align}
	\sigma G \sigma^{-1}: \quad 
	\begin{pmatrix}
	f \\ g\\ h
	\end{pmatrix}
	& \quad \longmapsto \quad 
	\left[
	\begin{pmatrix} 1 \\ 0 \\ 0 \end{pmatrix}, 
	\left[
	\begin{pmatrix} 0 \\ t \\ t-1 \end{pmatrix}, 
	\begin{pmatrix} f \\ g \\ h \end{pmatrix}
	\right]
	\right],\label{eq:G(xf+yg+zh)}\\
	\sigma H \sigma^{-1}: \quad 
	\begin{pmatrix}
	f \\ g\\ h
	\end{pmatrix}
	& \quad \longmapsto \quad
	\left[
	\begin{pmatrix} 0 \\ t \\ t-1 \end{pmatrix}, 
	\left[
	\begin{pmatrix} 1 \\ 0 \\ 0 \end{pmatrix}, 
	\begin{pmatrix} f \\ g \\ h \end{pmatrix}
	\right]
	\right] \label{eq:H(xf+yg+zh)}.
\end{align}
Evaluate the right-hand sides of \eqref{eq:G(xf+yg+zh)}, \eqref{eq:H(xf+yg+zh)} using Lemma \ref{v-val comp} and simplify the results to get \eqref{eq: op G}, \eqref{eq: op H}, respectively.
\end{proof}

\noindent
Now, we express the $\sigma$-images of $a_i$ and $b_i$ in vector-valued form.

\begin{lemma}\label{lem:G,H U_i}
For $i \in \mathbb{N}_0$, the map $\sigma$ sends
\begin{align}
	\label{eq:Gk(x_12)}
	a_{2i}
	& \quad \longmapsto \quad (-1)^i4^i\begin{pmatrix} (2t-1)^i \\ t(2t-1)^{i-1} \\ (t-1)(2t-1)^{i-1} \end{pmatrix}, \\
	\label{eq:Hk(x_03)}
	b_{2i}
	& \quad \longmapsto \quad (-1)^i4^i\begin{pmatrix} (2t-1)^{i-1} \\ t(2t-1)^i \\ (t-1)(2t-1)^{i} \end{pmatrix}, \\
	\label{eq:[x_03,Gk]}
	a_{2i+1}
	& \quad \longmapsto \quad 2(-1)^{i+1}4^i \begin{pmatrix} (2t-1)^i \\ t(2t-1)^{i} \\ (1-t)(2t-1)^{i} \end{pmatrix},\\
	\label{eq:[x_12,Hk]}
	b_{2i+1}
	& \quad \longmapsto \quad 2(-1)^i4^i\begin{pmatrix} (2t-1)^i \\ t(2t-1)^{i} \\ (1-t)(2t-1)^{i} \end{pmatrix},
\end{align}
where we interpret $(2t-1)^{-1}:=0$.
\end{lemma}
\begin{proof}
First, we show the equation \eqref{eq:Gk(x_12)}.
Since $(a_{2i})^\sigma = (G^i(x_{12}))^\sigma = \sigma {G}^i \sigma^{-1} (x_{12}^\sigma)$, we will evaluate $\sigma {G}^i \sigma^{-1}(x_{12}^\sigma)$. 
Recalling that $x_{12}^\sigma = x \otimes 1 = (1,0,0)^\top$ and using \eqref{eq: op G}, we find $\sigma{G}\sigma^{-1} (x_{12}^\sigma) = (2t-1, t, t-1)^\top$. 
Using induction on $i$, evaluate $\sigma {G}^i \sigma^{-1} (x_{12}^\sigma)$ and simplify the result to get the right-hand side of \eqref{eq:Gk(x_12)}.
We have shown \eqref{eq:Gk(x_12)}. 
Similarly, we derive \eqref{eq:Hk(x_03)}--\eqref{eq:[x_12,Hk]} in a comparable manner.
\end{proof}

\begin{example} 
We display the elements $a_{2i}^\sigma$, $b_{2i}^\sigma$, $a_{2i+1}^\sigma$, $b_{2i+1}^\sigma$ for $0\leq i \leq 4$:

{\renewcommand{\arraystretch}{1.3}
\begin{equation}
\begin{tabular}{c|cc}
$\quad i \quad $ & $a^\sigma_{2i}$ & $b^\sigma_{2i}$ \\  
\hline
$0$ 
	& $\begin{pmatrix}1 \\[-0.5em] 0\\[-0.5em] 0 \end{pmatrix}$ & $\begin{pmatrix} 0 \\[-0.5em] t \\[-0.5em] t-1 \end{pmatrix}$ \\
\hline
$1$
	& $-4\begin{pmatrix} 2t-1 \\[-0.5em] t \\[-0.5em] t-1 \end{pmatrix}$
	& $-4\begin{pmatrix} 1\\[-0.5em] t(2t-1) \\[-0.5em] (t-1)(2t-1) \end{pmatrix}$ \\
\hline
$2$
	& $4^2\begin{pmatrix} (2t-1)^2\\[-0.5em] t(2t-1) \\[-0.5em] (t-1)(2t-1) \end{pmatrix}$  
	& $4^2\begin{pmatrix} 2t-1\\[-0.5em] t(2t-1)^2 \\[-0.5em] (t-1)(2t-1)^2 \end{pmatrix}$ \\
\hline
$3$
	& $-4^3\begin{pmatrix} (2t-1)^3\\[-0.5em] t(2t-1)^2 \\[-0.5em] (t-1)(2t-1)^2 \end{pmatrix}$ 
	& $-4^3\begin{pmatrix} (2t-1)^2\\[-0.5em] t(2t-1)^3 \\[-0.5em] (t-1)(2t-1)^3 \end{pmatrix}$ \\
\hline
$4$
	& $4^4\begin{pmatrix} (2t-1)^4\\[-0.5em] t(2t-1)^3 \\[-0.5em] (t-1)(2t-1)^3 \end{pmatrix}$  
	& $4^4\begin{pmatrix} (2t-1)^3\\[-0.5em] t(2t-1)^4 \\[-0.5em] (t-1)(2t-1)^4 \end{pmatrix}$ 
\end{tabular},
\end{equation}
and
\begin{equation}
\begin{tabular}{c|cc}
$\quad i \quad $ & $a^\sigma_{2i+1}$ & $b^\sigma_{2i+1}$ \\  
\hline
$0$ 
	& $-2\begin{pmatrix} 1 \\[-0.5em] t \\[-0.5em] 1-t \end{pmatrix}$ & $2 \begin{pmatrix} 1 \\[-0.5em] t \\[-0.5em] 1-t\end{pmatrix}$\\
\hline
$1$
	& $2\cdot4\begin{pmatrix} 2t-1\\[-0.5em] t(2t-1) \\[-0.5em] (1-t)(2t-1) \end{pmatrix}$
	& $-2\cdot4\begin{pmatrix} 2t-1\\[-0.5em] t(2t-1) \\[-0.5em] (1-t)(2t-1) \end{pmatrix}$  \\
\hline
$2$
	& $-2\cdot4^2\begin{pmatrix} (2t-1)^2\\[-0.5em] t(2t-1)^2 \\[-0.5em] (1-t)(2t-1)^2 \end{pmatrix}$
	& $2\cdot4^2\begin{pmatrix} (2t-1)^2\\[-0.5em] t(2t-1)^2 \\[-0.5em] (1-t)(2t-1)^2 \end{pmatrix}$  \\
\hline
$3$
	& $2\cdot4^3\begin{pmatrix} (2t-1)^3\\[-0.5em] t(2t-1)^3 \\[-0.5em] (1-t)(2t-1)^3 \end{pmatrix}$
	& $-2\cdot4^3\begin{pmatrix} (2t-1)^3\\[-0.5em] t(2t-1)^3 \\[-0.5em] (1-t)(2t-1)^3 \end{pmatrix}$  \\
\hline
$4$
	& $-2\cdot4^4\begin{pmatrix} (2t-1)^4\\[-0.5em] t(2t-1)^4 \\[-0.5em] (1-t)(2t-1)^4 \end{pmatrix}$ 
	& $2\cdot4^4\begin{pmatrix} (2t-1)^4\\[-0.5em] t(2t-1)^4 \\[-0.5em] (1-t)(2t-1)^4 \end{pmatrix}$  
\end{tabular}.
\end{equation}
}
\end{example}

We now prove Lemma \ref{lem:subspace U_i}.

\begin{proof}[Proof of Lemma \ref{lem:subspace U_i}]
(i): By \eqref{eq:[x_03,Gk]} and  \eqref{eq:[x_12,Hk]}. \\
(ii): By \eqref{eq:Gk(x_12)} and \eqref{eq:Hk(x_03)}. \\
(iii): By construction, we have $\mathcal{U}_i = \sum^i_{j=0} U_j$.
By Lemma \ref{lem:G,H U_i}, we find that for each $i\geq 1$, $a_i, b_i \notin \mathcal{U}_{i-1}$.
This implies that $U_i \cap \mathcal{U}_{i-1}=\{0\}$. 
Therefore, the sum $\sum^i_{j=0} U_j$ is direct.\\
(iv): By part (iii), we have $\bigcup_{i\in \mathbb{N}_0} \mathcal{U}_i = \sum_{i \in \mathbb{N}_0} U_i$, which is a direct sum. 
By Lemma \ref{lem:O=unionU}, the result follows. 
\end{proof}

Consider the $\sigma$-image of $\mathcal{O}$.
Note that $\mathcal{O}^\sigma$ is the Lie subalgebra of $L(\sl2)^+$  generated by $x^\sigma_{12}$, $x^\sigma_{03}$.

\begin{proposition}\label{prop:basisOsigma}
The Lie algebra $\mathcal{O}^\sigma$ has a basis
\begin{equation}\label{eq:1st basis for O}
	a_{2i}^\sigma,\qquad b_{2i}^\sigma, \qquad  a_{2i+1}^\sigma, \qquad \qquad i\in \mathbb{N}_0.
\end{equation}
\end{proposition}
\begin{proof}
Recall the subspaces $\{U_i\}_{i\in \mathbb{N}_0}$ from Definition \ref{def:U_i}.
For each $i\in \mathbb{N}_0$, let $U^\sigma_i$ denote the subspace of $\mathcal{O}^\sigma$ that is the image of $U_i$ under $\sigma$.
By Lemma \ref{lem:subspace U_i}(i), (ii) and Lemma \ref{lem:G,H U_i}, the elements $a_{2i}^\sigma$, $b_{2i}^\sigma$ are a basis for $U_{2i}^\sigma$ and the element $a_{2i+1}^\sigma$ is a basis for $U_{2i+1}^\sigma$. 
Moreover, by \eqref{eq(2):sumUi=O} we have
\begin{equation}
	\mathcal{O}^\sigma = \sum_{i\in \mathbb{N}_0} U^\sigma_i \qquad (\text{\rm direct sum}).
\end{equation}
By these comments, the result follows.
\end{proof}

\medskip
We now turn our attention to the Onsager subalgebra $\mathcal{O}$ and present one of our main results.

\begin{theorem}\label{thm:1st basis for O}
The Onsager subalgebra $\mathcal{O}$ has a basis
\begin{equation}\label{eq:1st basis for O}
	a_{2i},\qquad b_{2i}, \qquad  a_{2i+1}, \qquad \qquad i\in \mathbb{N}_0.
\end{equation}
The Lie bracket acts on this basis as follows.
\begin{align}
	[a_{2i}, a_{2j}] & = \begin{cases}
		0 & \qquad \qquad \qquad  \text{if} \qquad i=j=0 \quad \text{or} \quad i\geq1, j\geq1; \label{thm1:eq[a2i,a2j]}\\
		4a_{2j-1} & \qquad \qquad \qquad \text{if} \qquad i=0, j\geq 1,
		\end{cases} \\
	[b_{2i}, a_{2j}] &= \begin{cases}
		a_{2i+1} & \text{if} \qquad i\geq0, j=0;\\
		a_{2j+1} & \text{if} \qquad i=0, j\geq1;\\
		a_{2i+2j+1} - 16 a_{2i+2j-3} & \text{if} \qquad i\geq1, j\geq1,
		\end{cases} \label{thm1:eq[b2i,a2j]}\\
	[b_{2i}, b_{2j}] & = \begin{cases}
		0 & \qquad \qquad \qquad \text{if} \qquad i=j=0 \quad \text{or} \quad i\geq1, j\geq1;\\
		4a_{2i-1} &\qquad \qquad \qquad \text{if} \qquad i\geq 1, j=0,
		\end{cases} \\
	[a_{2i}, a_{2j+1}] & = \begin{cases}
		a_{2j+2} & \quad \text{if} \qquad i=0, j\geq0;\\
		a_{2i+2j+2} + 4b_{2i+2j} &\quad  \text{if} \qquad i\geq1, j\geq0,
		\end{cases} \\
	[a_{2i+1}, b_{2j}] & = \begin{cases}
		b_{2i+2} &\quad  \text{if} \qquad i\geq0, j=0;\\
		b_{2i+2j+2} + 4a_{2i+2j} & \quad \text{if} \qquad i\geq0, j\geq1,
		\end{cases} \\
	[a_{2i+1}, a_{2j+1}] & = 0 \qquad \qquad \qquad \qquad \qquad \text{if} \qquad i\geq0, j\geq0. \label{thm1:eq[a2i+1,a2j+1]}
\end{align}
\end{theorem}
\begin{proof}
By Lemma \ref{lem:subspace U_i}(i), (ii) and (iv), the elements $a_{2i}$, $b_{2i}$, $ a_{2i+1}$ ($i\in \mathbb{N}_0$) form a basis for $\mathcal{O}$.
Next, we will prove \eqref{thm1:eq[a2i,a2j]}--\eqref{thm1:eq[a2i+1,a2j+1]}.
First, we show \eqref{thm1:eq[a2i,a2j]}.
Consider $[a^\sigma_{2i},a^\sigma_{2j}]$.
Evaluate this bracket using \eqref{eq:Gk(x_12)} and simplify the result by applying Lemma \ref{v-val comp}.
The result is zero if $i=j=0$ or $i\geq1$, $j\geq 1$, and it equals $4a^\sigma_{2j-1}$ if $i=0$, $j\geq 1$.
Since $\sigma$ is a Lie algebra isomorphism from $\tet \to L(\sl2)^+$, we obtain \eqref{thm1:eq[a2i,a2j]}.
Similarly, we routinely obtain \eqref{thm1:eq[b2i,a2j]}--\eqref{thm1:eq[a2i+1,a2j+1]}.
\end{proof}

\section{The subspace $X_{ij}$ of $\tet$}\label{Sec:x_ij-like O}

Let $X_{ij}$ be the subspace of $\tet$ as in Definition \ref{del:X_ij}, where $(i,j) \in \{(1,2),(2,3),(3,1)\}$.
Our goal of this section is to determine a basis for $X_{ij}$. 
To achieve this, we will show that the following sum is direct:
\begin{equation*}
	X_{ij} = (X_{ij} \cap \mathcal{O}) + (X_{ij} \cap \mathcal{O}') + (X_{ij} \cap \mathcal{O}'').
\end{equation*}
We also find a basis for each summand.
To do these things, we will work with $\sigma$-image of $X_{ij}$, where $\sigma$ is the Lie algebra isomorphism $\sigma: \tet \to L(\sl2)^+$ from Lemma \ref{def: sigma}.

\begin{lemma}\label{lem:xij-like elts in L(sl2)+}
Recall the basis \eqref{eq:1st basis for O} for $\mathcal{O}$.
For $i\in \mathbb{N}_0$, the following {\rm(i)}--{\rm(iii)} hold. 
\begin{enumerate}[label=(\roman*), font=\normalfont]
\item The element $a_{2i} + 4 b_{2i-2}$ is $x_{12}$-like. 
Moreover, its image under the map $\sigma$ is 
	\begin{equation}\label{eq:x12 in O}
	(-1)^i4^i\begin{pmatrix} (2t-1)^i - (2t-1)^{i-2} \\ 0 \\ 0 \end{pmatrix},
	\end{equation}
	where we interpret $b_{-2}:=0$, $(2t-1)^{-1}:=0$, and $(2t-1)^{-2}:=0$.
\item The element $a_{2i+1} + 2a_{2i}+4a_{2i-1}-2b_{2i}$ is $x_{23}$-like.
Moreover, its image under the map $\sigma$ is
	\begin{equation}\label{eq:x23 in O}
	(-1)^{i+1}4^{i+1} \begin{pmatrix} 0 \\ t(2t-1)^{i} - t(2t-1)^{i-1} \\ 0 \end{pmatrix},
	\end{equation}
	where we interpret $a_{-1}:=0$ and $(2t-1)^{-1}:=0$.
\item The element $a_{2i+1} + 2a_{2i} - 4a_{2i-1} + 2b_{2i}$ is $x_{31}$-like. 
Moreover, its image under the map $\sigma$ is
	\begin{equation}\label{eq:x31 in O}
	(-1)^{i+1}4^{i+1} \begin{pmatrix} 0 \\ 0 \\  (1-t)(2t-1)^i+(1-t)(2t-1)^{i-1}  \end{pmatrix},
	\end{equation}
	where we interpret $a_{-1}:=0$ and $(2t-1)^{-1}:=0$.
\end{enumerate}
\end{lemma}
\begin{proof}
(i): We first show the second assertion.
Evaluate $\left(a_{2i} + 4 b_{2i-2}\right)^\sigma$ using \eqref{eq:Gk(x_12)}, \eqref{eq:Hk(x_03)} to get the second assertion:
\begin{align*}
	\left(a_{2i} + 4 b_{2i-2}\right)^\sigma
	& = (-1)^i4^i\begin{pmatrix} (2t-1)^i \\ t(2t-1)^{i-1} \\ (t-1)(2t-1)^{i-1} \end{pmatrix}
		+ (-1)^{i-1}4^i \begin{pmatrix} (2t-1)^{i-2} \\ t(2t-1)^{i-1} \\ (t-1)(2t-1)^{i-1} \end{pmatrix} \\
	& = 	(-1)^{i}4^{i}\begin{pmatrix} (2t-1)^i-(2t-1)^{i-2} \\ 0 \\ 0 \end{pmatrix}.
\end{align*}
Next, by Lemma \ref{lem:x_ij-elts L(sl2)+} the second assertion implies $\left(a_{2i} + 4 b_{2i-2}\right)^\sigma \in X^\sigma_{12}$.
Hence, we showed the first assertion.\\
\noindent
(ii): Evaluate $(a_{2i+1} + 2a_{2i}  + 4a_{2i-1}-2b_{2i})^\sigma$ using \eqref{eq:Gk(x_12)}--\eqref{eq:[x_03,Gk]} to get the second assertion:
\begin{align*}
	(a_{2i+1} + 2a_{2i} & + 4a_{2i-1}-2b_{2i})^\sigma \\
	= & \ 2(-1)^{i+1}4^i \begin{pmatrix}  (2t-1)^i \\ t(2t-1)^{i} \\ (1-t)(2t-1)^{i} \end{pmatrix}
	+ 2(-1)^i 4^i \begin{pmatrix}  (2t-1)^{i} \\ t(2t-1)^{i-1} \\ (t-1)(2t-1)^{i-1} \end{pmatrix} \\
	 & \qquad  +2(-1)^i4^i\begin{pmatrix} (2t-1)^{i-1} \\ t(2t-1)^{i-1} \\ (1-t)(2t-1)^{i-1} \end{pmatrix}
	+ 2(-1)^{i+1}4^i \begin{pmatrix}  (2t-1)^{i-1} \\ t(2t-1)^i \\ (t-1)(2t-1)^{i} \end{pmatrix} \\
	= & \ (-1)^{i+1}4^{i+1} \begin{pmatrix}  0 \\ t(2t-1)^i - t(2t-1)^{i-1} \\ 0 \end{pmatrix}.
\end{align*}
Next, by Lemma \ref{lem:x_ij-elts L(sl2)+} the second assertion implies $(a_{2i+1} + 2a_{2i} + 4a_{2i-1}-2b_{2i})^\sigma \in X^\sigma_{23}$.
Hence, we showed the first assertion.\\
\noindent
(iii): Evaluate $(a_{2i+1} + 2a_{2i} - 4a_{2i-1} + 2b_{2i})^\sigma$ using \eqref{eq:Gk(x_12)}--\eqref{eq:[x_03,Gk]} to get the second assertion:
\begin{align*}
	(a_{2i+1} + 2a_{2i} & - 4a_{2i-1} + 2b_{2i})^\sigma \\
	= & \ 2(-1)^{i+1}4^i \begin{pmatrix}  (2t-1)^i \\ t(2t-1)^{i} \\ (1-t)(2t-1)^{i}  \end{pmatrix}
	+ 2(-1)^i4^i \begin{pmatrix} (2t-1)^{i} \\ t(2t-1)^{i-1} \\ (t-1)(2t-1)^{i-1}  \end{pmatrix} \\
	 & \qquad  +2(-1)^{i+1}4^i \begin{pmatrix} (2t-1)^{i-1} \\ t(2t-1)^{i-1} \\ (1-t)(2t-1)^{i-1}  \end{pmatrix}
	+ 2(-1)^i4^i \begin{pmatrix}  (2t-1)^{i-1} \\ t(2t-1)^i \\ (t-1)(2t-1)^{i}  \end{pmatrix} \\
	= & \ (-1)^{i+1}4^{i+1} \begin{pmatrix} 0 \\ 0 \\ (1-t)(2t-1)^{i} + (1-t)(2t-1)^{i-1}  \end{pmatrix}.
\end{align*}
Next, by Lemma \ref{lem:x_ij-elts L(sl2)+} the second assertion implies $(a_{2i+1} + 2a_{2i} - 4a_{2i-1} + 2b_{2i})^\sigma \in X^\sigma_{31}$.
Hence, we showed the first assertion.
\end{proof}

\begin{definition}\label{Def:x_ij like in O}
Referring to Lemma \ref{lem:xij-like elts in L(sl2)+}, we define
\begin{align}
	\boldsymbol{x}_{i} & = a_{2i} + 4 b_{2i-2},\label{eq:a(i) tet}\\
	\boldsymbol{y}_{i} & = a_{2i+1} + 2a_{2i}+4a_{2i-1}-2b_{2i},\label{eq:b(i) tet}\\
	\boldsymbol{z}_{i} & = a_{2i+1} + 2a_{2i} - 4a_{2i-1} + 2b_{2i},\label{eq:c(i) tet}
\end{align}
for $i\in \mathbb{N}_0$, where $a_{-1}:=0$ and $b_{-2}:=0$.
\end{definition}
By Lemma \ref{lem:subspace U_i}(i),(ii) we observe that 
\begin{equation}\label{eq:x,y,z Ui}
	\bx_i  \in U_{2i-2} + U_{2i}, \quad 
	\by_{i}  \in U_{2i-1} + U_{2i} + U_{2i+1}, \quad
	\bz_{i}  \in U_{2i-1} + U_{2i} + U_{2i+1}, 
\end{equation}
where $U_{-1}:=0$, $U_{-2}:=0$.

\begin{notation}\label{notation:x,y,z}
For the rest of this paper, we retain the notation $\bx_i$, $\by_{i}$, $\bz_{i}$ for their $\sigma$-images.
\end{notation}

\begin{lemma}\label{lem:x,y,z F[t]}
The following {\rm{(i)}}--{\rm{(iii)}} hold.
\begin{enumerate}[label=(\roman*), font=\normalfont]
	\item $x \otimes \mathbb{F}[t]$ has a basis $\{\bx_{i} \mid i \in \mathbb{N}_0\}$.
	\item $y \otimes t\mathbb{F}[t]$ has a basis $\{\by_{i} \mid i \in \mathbb{N}_0\}$.
	\item $z \otimes (1-t)\mathbb{F}[t]$ has a basis $\{\bz_{i} \mid i \in \mathbb{N}_0\}$.
\end{enumerate}
\end{lemma}
\begin{proof}
(i): From \eqref{eq:x12 in O}, $\bx_{i}$ has the form
$$
	 x\otimes \left( (-1)^i8^{i}t^i + \text{(lower-degree terms)}\right).
$$
This implies that the set $\{\bx_{i} \mid i \in \mathbb{N}_0\}$ is linearly independent and spans $x\otimes\mathbb{F}[t]$. Therefore, $\{\bx_{i} \mid i \in \mathbb{N}_0\}$ forms a basis for $x\otimes\mathbb{F}[t]$. \\
(ii): From \eqref{eq:x23 in O}, $\by_{i}$ has the form
$$
	 y\otimes t\left( (-1)^{i+1}4\cdot8^{i}t^i + \text{(lower-degree terms)}\right).
$$
Therefore, $\{\by_{i} \mid i \in \mathbb{N}_0\}$ forms a basis for $y\otimes t\mathbb{F}[t]$. \\
(iii): From \eqref{eq:x31 in O}, $\bz_{i}$ has the form
$$
	 z\otimes (1-t) \left( (-1)^{i+1}4\cdot8^{i}t^i + \text{(lower-degree terms)}\right).
$$
Therefore, $\{\bz_{i} \mid i \in \mathbb{N}_0\}$ forms a basis for $z\otimes (1-t)\mathbb{F}[t]$. 
\end{proof}

\begin{lemma}[cf. {\cite[Corollary 11.3]{2007HarTer}}] \label{lem: decomp O}
The Lie algebra $\mathcal{O}^\sigma$ satisfies
\begin{equation}\label{decomp O sigma}
	\mathcal{O}^\sigma = x\otimes \mathbb{F}[t] + y \otimes t \mathbb{F}[t] + z\otimes(1-t)\mathbb{F}[t] \qquad  (\text{\rm direct sum}).
\end{equation}
\end{lemma}
\begin{proof}
Let $\mathcal{L}$ denote the right-hand side of \eqref{decomp O sigma}.
We show that $\mathcal{O}^\sigma = \mathcal{L}$.
By Lemma \ref{lem:G,H U_i}, it readily follows that $\mathcal{O}^\sigma \subseteq \mathcal{L}$.
To get the reverse inclusion, according to Lemma \ref{lem:x,y,z F[t]} we have $\mathcal{L} = \operatorname{Span}\{\bx_{i},\by_{i},\bz_{i} \mid i\in \mathbb{N}_0\}$.
Using \eqref{eq:a(i) tet}--\eqref{eq:c(i) tet} and considering the vectors $a^\sigma_{2i}$, $b^\sigma_{2i}$, $a^\sigma_{2i+1}$ $(i\in \mathbb{N}_0)$ that form a basis for $\mathcal{O}^\sigma$, it follows that $\mathcal{O}^\sigma \supseteq \mathcal{L}$.
Thus, we have shown $\mathcal{O}^\sigma=\mathcal{L}$.
By Lemma \ref{lem:x,y,z F[t]}, the sum \eqref{decomp O sigma} is direct.
\end{proof}

\begin{proposition}\label{prop:Osigma}
The Lie algebra $\mathcal{O}^\sigma$ has a basis
\begin{equation}
	\bx_{i}, \qquad \by_{i}, \qquad \bz_{i}, \qquad \qquad i\in \mathbb{N}_0.
\end{equation}
\end{proposition}
\begin{proof}
By Lemmas \ref{lem:x,y,z F[t]} and \ref{lem: decomp O}, the result follows.
\end{proof}

\medskip
Recall the subspace $X^\sigma_{ij}$ of $L(\sl2)^+$.
Using \eqref{def:eq[u,v]x(ab)} and Lemma \ref{lem:L_ij like}, we see that $X_{ij}^\sigma$ forms a subalgebra of $L(\sl2)^+$.
Thus, $X^\sigma_{ij}\cap\mathcal{O}^\sigma$ also forms a subalgebra of $L(\sl2)^+$.

\begin{lemma}\label{lem:Lij cap O}
We have
\begin{equation}\label{eq:Xij sigma}
	X^\sigma_{12}\cap \mathcal{O}^\sigma = x\otimes \mathbb{F}[t], \qquad
	X^\sigma_{23}\cap \mathcal{O}^\sigma = y\otimes t\mathbb{F}[t], \qquad
	X^\sigma_{31}\cap \mathcal{O}^\sigma = z\otimes (1-t)\mathbb{F}[t].
\end{equation}
Moreover, the following {\rm{(i)}}--{\rm{(iii)}} hold.
\begin{enumerate}[label=(\roman*), font=\normalfont]
	\item $X_{12}^\sigma\cap\mathcal{O}^\sigma$ has a basis $\{\bx_{i} \mid i \in \mathbb{N}_0\}$.
	\item $X_{23}^\sigma\cap\mathcal{O}^\sigma$ has a basis $\{\by_{i} \mid i \in \mathbb{N}_0\}$.
	\item $X_{31}^\sigma\cap\mathcal{O}^\sigma$ has a basis $\{\bz_{i} \mid i \in \mathbb{N}_0\}$.
\end{enumerate}
\end{lemma}
\begin{proof}
Line \eqref{eq:Xij sigma} follows from Lemma \ref{lem:x_ij-elts L(sl2)+} and \eqref{decomp O sigma}.
Apply Lemma \ref{lem:x,y,z F[t]} to \eqref{eq:Xij sigma} to obtain (i)--(iii).
\end{proof}

\noindent
Recall the automorphism $\prime$ of $\L$ from Lemma \ref{lem:auto L(sl2)+}.
Consider the elements $\bx'_i,\by'_{i},\bz'_{i}$ $(i\in \mathbb{N}_0)$ in $\mathcal{O}'$ and $\bx''_i,\by''_{i},\bz''_{i}$ $(i\in \mathbb{N}_0)$ in $\mathcal{O}''$, as presented in Lemma \ref{B.lem:bxbybz}.

\begin{lemma}\label{lem:X_12, O,O',O''} 
We have
\begin{equation}\label{eq:X12 sigma + O,O',O''}
	X^\sigma_{12}\cap \mathcal{O}^\sigma = x\otimes \mathbb{F}[t], \qquad
	X^\sigma_{12}\cap \mathcal{O}^{\prime\sigma} = x\otimes (1-t')\mathbb{F}[t'], \qquad
	X^\sigma_{12}\cap \mathcal{O}^{\prime\prime\sigma} = x\otimes t''\mathbb{F}[t''].
\end{equation}
Moreover, the following {\rm{(i)}}--{\rm{(iii)}} hold.
\begin{enumerate}[label=(\roman*), font=\normalfont]
	\item $X_{12}^\sigma\cap\mathcal{O}^\sigma$ has a basis $\{\bx_{i} \mid i \in \mathbb{N}_0\}$.
	\item $X_{12}^\sigma\cap\mathcal{O}^{\prime\sigma}$ has a basis $\{\bz'_{i} \mid i \in \mathbb{N}_0\}$.
	\item $X_{12}^\sigma\cap\mathcal{O}^{\prime\prime\sigma}$ has a basis $\{\by''_{i} \mid i \in \mathbb{N}_0\}$.
\end{enumerate}
\end{lemma}
\begin{proof}
Applying the map $\prime$ to the third equation in \eqref{eq:Xij sigma} and Lemma \ref{lem:Lij cap O}(iii), we obtain the second in \eqref{eq:X12 sigma + O,O',O''} and (ii).
Similarly, applying the map $\prime\prime$ to the second equation in \eqref{eq:Xij sigma} and Lemma \ref{lem:Lij cap O}(ii), we obtain the third in \eqref{eq:X12 sigma + O,O',O''} and (iii).
The result follows.
\end{proof}
\begin{remark}
Applying the maps $\prime$ and $\prime\prime$ to Lemma \ref{lem:X_12, O,O',O''} yields similar results for $X^\sigma_{23}$ and $X^\sigma_{31}$.
\end{remark}
\begin{lemma}\label{lem:Osigma DS}
The following sums are direct:
\begin{align}
	\mathcal{O}^{\sigma} 
	&= (X^\sigma_{12} \cap \mathcal{O}^{\sigma}) + (X^\sigma_{23} \cap \mathcal{O}^{\sigma}) + (X^\sigma_{31} \cap \mathcal{O}^{\sigma}), \label{eq(1):O sigma DS}\\
	\mathcal{O}^{\prime\sigma} 
	&= (X^\sigma_{12} \cap \mathcal{O}^{\prime\sigma}) + (X^\sigma_{23} \cap \mathcal{O}^{\prime\sigma}) + (X^\sigma_{31} \cap \mathcal{O}^{\prime\sigma}), \label{eq(2):O sigma DS}\\
	\mathcal{O}^{\prime\prime\sigma} 
	&= (X^\sigma_{12} \cap \mathcal{O}^{\prime\prime\sigma}) + (X^\sigma_{23} \cap \mathcal{O}^{\prime\prime\sigma}) + (X^\sigma_{31} \cap \mathcal{O}^{\prime\prime\sigma}).\label{eq(3):O sigma DS}
\end{align}
\end{lemma}
\begin{proof}
Equation \eqref{eq(1):O sigma DS} follows from \eqref{decomp O sigma} and \eqref{eq:Xij sigma}.
Apply the maps $\prime$ and $\prime\prime$ to \eqref{eq(1):O sigma DS} to get equations \eqref{eq(2):O sigma DS} and \eqref{eq(3):O sigma DS}, respectively.
\end{proof}

\begin{lemma}\label{lem:X_ij DS}
The following sums are direct:
\begin{align}
	X^\sigma_{12} & = (X^\sigma_{12} \cap \mathcal{O}^\sigma) + (X^\sigma_{12} \cap \mathcal{O}^{\prime\sigma}) + (X^\sigma_{12} \cap \mathcal{O}^{\prime\prime\sigma}), \label{eq: X12 DS}\\
	X^\sigma_{23} & = (X^\sigma_{23} \cap \mathcal{O}^\sigma) + (X^\sigma_{23} \cap \mathcal{O}^{\prime\sigma}) + (X^\sigma_{23} \cap \mathcal{O}^{\prime\prime\sigma}), \label{eq: X23 DS}\\
	X^\sigma_{31} & = (X^\sigma_{31} \cap \mathcal{O}^\sigma) + (X^\sigma_{31} \cap \mathcal{O}^{\prime\sigma}) + (X^\sigma_{31} \cap \mathcal{O}^{\prime\prime\sigma}). \label{eq: X31 DS}	
\end{align}
\end{lemma}
\begin{proof}
By \eqref{eq:X12-DS} and \eqref{eq:X12 sigma + O,O',O''}, we obtain \eqref{eq: X12 DS}.
Apply the maps $\prime$ and $\prime \prime$ to \eqref{eq: X12 DS} to obtain \eqref{eq: X12 DS} and \eqref{eq: X12 DS}, respectively.
The result follows.
\end{proof}

We now return our attention to the Lie algebra $\tet$.
In Lemmas \ref{lem:Lij cap O} and \ref{lem:X_ij DS}, we derived results related to $X^\sigma_{12}$, $X^\sigma_{23}$, $X^\sigma_{31}$ in $\L$.
Pulling these results back to $\tet$ using the map $\sigma$ yields the following results.
\begin{proposition}\label{Prop:direct sums}
The following sums are direct:
\begin{align}
	\mathcal{O}
		&= (X_{12} \cap \mathcal{O}) + (X_{23} \cap \mathcal{O}) + (X_{31} \cap \mathcal{O}), \label{eq(1):O DS}\\
	\mathcal{O}'
		&= (X_{23} \cap \mathcal{O'}) + (X_{31} \cap \mathcal{O}') + (X_{12} \cap \mathcal{O}'), \label{eq(2):O DS}\\
	\mathcal{O}''
		&= (X_{31} \cap \mathcal{O}'') + (X_{12} \cap \mathcal{O}'') + (X_{23} \cap \mathcal{O}''). \label{eq(3):O DS}
\end{align}
Moreover, the following sums are direct:
\begin{align}
	X_{12} & = (X_{12} \cap \mathcal{O}) + (X_{12} \cap \mathcal{O}^{\prime}) + (X_{12} \cap \mathcal{O}^{\prime\prime}), \\
	X_{23} & = (X_{23} \cap \mathcal{O}) + (X_{23} \cap \mathcal{O}^{\prime}) + (X_{23} \cap \mathcal{O}^{\prime\prime}), \\
	X_{31} & = (X_{31} \cap \mathcal{O}) + (X_{31} \cap \mathcal{O}^{\prime}) + (X_{31} \cap \mathcal{O}^{\prime\prime}). 
\end{align}
\end{proposition}
\begin{proof}
By Lemmas \ref{lem:X_12, O,O',O''} and \ref{lem:X_ij DS}.
\end{proof}

We summarize the results of Proposition \ref{Prop:direct sums} in the following $3\times 3$ matrix:
\begin{equation}\label{matrix(1): X_ij}
{\renewcommand{\arraystretch}{1.5}
\begin{array}{c|ccccc}
\tet	& \mathcal{O} & & \mathcal{O}' & & \mathcal{O}'' \\
\hline
X_{12}	& X_{12}\cap \mathcal{O} & & X_{12}\cap \mathcal{O'} & & X_{12}\cap \mathcal{O''} \\
X_{23}	& X_{23}\cap \mathcal{O} & & X_{23}\cap \mathcal{O'} & & X_{23}\cap \mathcal{O''} \\
X_{31}	& X_{31}\cap \mathcal{O} & & X_{31}\cap \mathcal{O'} & & X_{31}\cap \mathcal{O''} 
\end{array}}
\end{equation}
This matrix illustrates:
\begin{itemize}[itemsep=0pt]
	\item In each row, the sum of the three entries is direct and equal to the subspace on the left.
	\item In each column, the sum of the three entries is direct and equal to the subspace on the top.
	\item The sum of all nine entries is direct and equal to the Lie algebra $\tet$.	
\end{itemize}

\begin{theorem}\label{thm:Xij basis}
The following {\rm(i)}--{\rm(iv)} hold.
\begin{enumerate}[label=(\roman*), font=\normalfont]
	\item The subspace $X_{12}$ has a basis
	\begin{equation}\label{eq:basis X_12}
	\bx_{i}, \qquad \bz'_{i}, \qquad \by''_{i} \qquad \qquad  (i \in \mathbb{N}_0).
	\end{equation}
	\item The subspace $X_{23}$ has a basis
	\begin{equation}\label{eq:basis X_23}
	\by_{i}, \qquad \bx'_{i}, \qquad \bz''_{i} \qquad \qquad  (i \in \mathbb{N}_0).
	\end{equation}
	\item The subspace $X_{31}$ has a basis
	\begin{equation}\label{eq:basis X_31}
	\bz_{i}, \qquad \by'_{i}, \qquad \bx''_{i} \qquad \qquad  (i \in \mathbb{N}_0).
	\end{equation}
	\item The union of \eqref{eq:basis X_12}--\eqref{eq:basis X_31} forms a basis for $\tet$.
\end{enumerate}
\end{theorem}
\begin{proof}
(i)--(iii): Use Lemma \ref{lem:Lij cap O}(i)--(iii) and Lemma \ref{lem:X_12, O,O',O''}(i)--(iii) along with the automorphism $\prime$.\\
(iv): By \eqref{cor:tet=ds(Xij)}.
\end{proof}

\noindent
Theorem \ref{thm:Xij basis} is summarized by the following $3\times 3$ matrix:
\begin{equation}\label{matrix(2): X_ij}
{\renewcommand{\arraystretch}{1.5}
\begin{array}{c|ccccc}
\tet	& \mathcal{O} & & \mathcal{O}' & & \mathcal{O}'' \\
\hline
X_{12}	& \{\bx_{i} \mid i \in \mathbb{N}_0\} & & \{\bz'_{i} \mid i \in \mathbb{N}_0\} & & \{\by''_{i} \mid i \in \mathbb{N}_0\} \\
X_{23}	& \{\by_{i} \mid i \in \mathbb{N}_0\} & & \{\bx'_{i} \mid i \in \mathbb{N}_0\} & & \{\bz''_{i} \mid i \in \mathbb{N}_0\} \\
X_{31}	& \{\bz_{i} \mid i \in \mathbb{N}_0\} & & \{\by'_{i} \mid i \in \mathbb{N}_0\} & & \{\bx''_{i} \mid i \in \mathbb{N}_0\} 
\end{array}}
\end{equation}
This matrix illustrates:
\begin{itemize}[itemsep=0pt]
	\item For every entry, the vectors form a basis for the corresponding entry in matrix \eqref{matrix(1): X_ij}.
	\item In each row, the union of the vectors shown in the three entries forms a basis for the subspace on the left.
	\item In each column, the union of the vectors shown in the three entries forms a basis for the subspace on the top.
	\item The union of the vectors shown in all nine entries forms a basis for $\tet$.
\end{itemize}

\section{A second basis for $\mathcal{O}$}\label{sec:2nd basis O}

In Theorem \ref{thm:1st basis for O}, we introduced a basis for $\mathcal{O}$. 
In this section, we consider another basis of $\mathcal{O}$, consisting of $x_{ij}$-like elements introduced in Definition \ref{Def:x_ij like in O}, and discuss how the Lie bracket acts on this basis.

\begin{theorem}\label{2nd main thm}
The Onsager subalgebra $\mathcal{O}$ has a basis
\begin{equation}\label{2nd basis for O}
	\bx_{i}, \qquad \by_{i}, \qquad \bz_{i} \qquad \qquad (i\in \mathbb{N}_0),
\end{equation}
such that $x_{12}=\bx_{0}$ and $x_{03}=(\bz_{0}-\by_{0})/4$.
The Lie bracket acts on this basis as follows.
\begin{align}
	&&[\bx_{i}, \bx_{j}] & = 0 \qquad \qquad (i,j \geq 0), &&\label{eq: [x,x]}\\
	&&[\by_{i}, \by_{j}] & = 0 \qquad \qquad (i,j \geq 0), &&\label{eq: [y,y]}\\
	&&[\bz_{i}, \bz_{j}] & = 0 \qquad \qquad (i,j \geq 0). &&\label{eq: [z,z]}
\end{align}
Moreover,
\begin{align}	
	&&[\bx_{0}, \by_{0}] & = \bx_{1} + 2\by_{0} - 4\bx_{0}, \label{eq: [x,y] 0,0}\\	
	&&[\bx_{1}, \by_{0}] & = \bx_{2} + 2\by_{1} - 4\bx_{1} - 8\by_{0} + 16\bx_{0}, \label{eq: [x,y] 1,0}\\
	&&[\bx_{i}, \by_{0}] & = \bx_{i+1} + 2\by_{i} - 4\bx_{i} - 8\by_{i-1}  && (i\geq 2), &&\\	
	&&[\bx_{i}, \by_{j}] & = \bx_{i+j+1} + 2\by_{i+j} &&  (i = 0,1; \ j\geq 1),\\		
	&&[\bx_{i}, \by_{j}] & = \bx_{i+j+1} + 2\by_{i+j} - 16 \bx_{i+j-1} - 32\by_{i+j-2} && (i \geq 2, j\geq 1), && \label{eq: [x,y] ij geq 21}	
\end{align}
and
\begin{align}	
	&&[\bx_{0}, \bz_{0}] & = \bx_{1} - 2\bz_{0} + 4\bx_{0}, \label{eq:[x,z] 0,0}\\	
	&&[\bx_{1}, \bz_{0}] & = \bx_{2} - 2\bz_{1} + 4\bx_{1} - 8\bz_{0} + 16\bx_{0}, \label{eq:[x,z] 1,0}\\
	&&[\bx_{i}, \bz_{0}] & = \bx_{i+1} - 2\bz_{i} +  4 \bx_{i} - 8\bz_{i-1} && (i\geq 2), \\
	&&[\bx_{i}, \bz_{j}] & = \bx_{i+j+1} - 2\bz_{i+j} &&  (i=0,1; \ j\geq 1), \\		
	&&[\bx_{i}, \bz_{j}] & = \bx_{i+j+1}-2\bz_{i+j} - 16 \bx_{i+j-1} + 32 \bz_{i+j-2} && (i\geq 2, j \geq 1), && \label{eq:[x,z] ij geq 21} 
\end{align}
and
\begin{align}
	&&[\by_{0}, \bz_{0}] & = \bz_{1} - \by_{1}, \label{eq: [y,z] 0,0}\\	
	&&[\by_{i}, \bz_{0}] & = \bz_{i+1} - \by_{i+1} + 4 \bz_{i} - 4\by_{i} &&  (i\geq 1), &&\\	
	&&[\by_{0}, \bz_{j}] & = \bz_{j+1} - \by_{j+1} - 4 \bz_{j} + 4 \by_{j}  &&  (j\geq 1), &&\\	
	&&[\by_{i}, \bz_{j}] & = \bz_{i+j+1}-\by_{i+j+1} - 16 \bz_{i+j-1} +16 \by_{i+j-1} &&  (i,j \geq 1).&& \label{eq: [y,z] i,j geq 1} 
\end{align}
\end{theorem}

\medskip
To prove this theorem, we introduce the polynomials $p_i \in \mathbb{F}[t]$ for computational convenience, defined as
\begin{equation}\label{eq:poly p_i}
	p_i := (2t-1)^i  \qquad i\in\mathbb{N}_0.
\end{equation}
Additionally, we set $p_{-1}:=0$, $p_{-2}:=0$.
It is readily observed that for $i,j \in \mathbb{N}_0$
\begin{gather}\label{eq:algebra p_i}
	p_ip_j  = p_{i+j}, \qquad 
	2tp_i  = p_i + p_{i+1}, \qquad 
	2(1-t)p_i  = p_i-p_{i+1}.
\end{gather}
Using the polynomials $p_i$, we express the vectors $\bx_{i}$, $\by_{i}$, $\bz_{i}$ for $i \in \mathbb{N}_0$ in \eqref{eq:x12 in O}--\eqref{eq:x31 in O} as follows:
\begin{gather}
	\bx_{i}  = (-1)^i4^i \begin{pmatrix} p_i - p_{i-2} \\ 0 \\ 0 \end{pmatrix}, \label{eq:bx(i)}\\
	\by_{i}  = (-1)^{i+1} 4^{i+1} \begin{pmatrix} 0 \\ t(p_i - p_{i-1}) \\ 0 \end{pmatrix}, \qquad \qquad
	\bz_{i}  = (-1)^{i+1} 4^{i+1} \begin{pmatrix} 0 \\ 0 \\ (1-t)(p_i + p_{i-1})   \end{pmatrix}. \label{eq:by,bz(i)}
\end{gather}

\begin{lemma}
We have
\begin{equation}\label{lem: bx pi form}
	\bx_{0}=\begin{pmatrix} 1 \\ 0 \\ 0 \end{pmatrix}, \qquad 
	\bx_{1}=-4\begin{pmatrix} 2t-1 \\ 0 \\ 0 \end{pmatrix}, \qquad
	\bx_{i} = (-1)^{i}4^{i+1}t(t-1)\begin{pmatrix} p_{i-2} \\ 0 \\ 0 \end{pmatrix},
\end{equation}
for $i=2,3,4,\ldots$, and
\begin{equation}\label{lem: by pi form}
	\by_{0}=-4\begin{pmatrix} 0 \\ t \\ 0 \end{pmatrix}, \qquad 
	\by_{i} = 2\cdot(-1)^{i+1}\cdot4^{i+1}t(t-1) \begin{pmatrix} 0 \\ p_{i-1} \\ 0 \end{pmatrix},
\end{equation}
for $i\in \mathbb{N}$, and 
\begin{equation}\label{lem: bz pi form}
	\bz_{0}=-4 \begin{pmatrix} 0 \\ 0 \\ 1-t \end{pmatrix}, \qquad  
	\bz_{i} = 2 \cdot (-1)^{i} \cdot 4^{i+1}t(t-1)\begin{pmatrix} 0 \\ 0 \\ p_{i-1}  \end{pmatrix},
\end{equation}
for $i\in \mathbb{N}$.
\end{lemma}
\begin{proof}
Evaluate \eqref{eq:bx(i)}, \eqref{eq:by,bz(i)} using \eqref{eq:poly p_i} and simplify them to obtain the results.
\end{proof}

We now prove the theorem.

\begin{proof}[Proof of Theorem \ref{2nd main thm}]
By Proposition \ref{prop:Osigma}, the elements $\bx_{i}$, $\by_{i}$, $\bz_{i}$ form a basis for $\mathcal{O}$.
From \eqref{eq:a(i) tet} at $i=0$, we have $x_{12}=\bx_{0}$.
From \eqref{eq:b(i) tet} and \eqref{eq:c(i) tet} at $i=0$, we have
\begin{align}
	\by_{0} & = [x_{03}, x_{12}] + 2x_{12} - 2x_{03},\\
	\bz_{0} & = [x_{03}, x_{12}] + 2x_{12} + 2x_{03}.
\end{align}
Solve these equations for $x_{03}$ to get $(\bz(0)-\by(0))/4$.
We now prove \eqref{eq: [x,x]}--\eqref{eq: [y,z] i,j geq 1}.
First, equations \eqref{eq: [x,x]}--\eqref{eq: [z,z]} follow from the fact that $[u,u]=0$ for all $u \in \sl2$. 
Next, to prove equations \eqref{eq: [x,y] 0,0}--\eqref{eq: [y,z] i,j geq 1}, we consider the images of each basis element in \eqref{2nd basis for O} under the map $\sigma$ and verify that their images satisfy equations \eqref{eq: [x,y] 0,0}--\eqref{eq: [y,z] i,j geq 1}.
This verification proceeds as follows:
\begin{enumerate}[label=(\roman*), font=\normalfont]
	\item Evaluate $[\bx_{i}, \by_{j}]$ using \eqref{lem: bx pi form}, \eqref{lem: by pi form}, and Lemma \ref{v-val comp}, and then simplify the results to obtain the expressions on the right-hand side in equations \eqref{eq: [x,y] 0,0}--\eqref{eq: [x,y] ij geq 21}.

	\item Evaluate $[\bx_{i}, \bz_{j}]$ using \eqref{lem: bx pi form}, \eqref{lem: bz pi form}, and Lemma \ref{v-val comp}, and then simplify the results to obtain the expressions on the right-hand side in equations \eqref{eq:[x,z] 0,0}--\eqref{eq:[x,z] ij geq 21}.

	\item Evaluate $[\by_{i}, \bz_{j}]$ using \eqref{lem: by pi form}, \eqref{lem: bz pi form}, and Lemma \ref{v-val comp}, and then simplify the results to obtain the expressions on the right-hand side in equations \eqref{eq: [y,z] 0,0}--\eqref{eq: [y,z] i,j geq 1}.
\end{enumerate}
Through steps (i)--(iii), it follows that equations \eqref{eq: [x,y] 0,0}--\eqref{eq: [y,z] i,j geq 1} are valid.
The proof is complete.
\end{proof}

\begin{lemma}
The basis elements in \eqref{2nd basis for O} for $\mathcal{O}$ are recursively obtained as follows.
Let
\begin{equation*}
	\bx_0 = x_{12}, \qquad 
	\by_0 = [x_{03},x_{12}] + 2x_{12} - 2x_{03}, \qquad
	\bz_0 = [x_{03},x_{12}] + 2x_{12} + 2x_{03}.
\end{equation*}
Then
\begin{align}
	\bx_1 & = [\bx_0, \by_0] - 2\by_0 + 4\bx_0 \label{pf:bx 1(1)}\\
		& = [\bx_0, \bz_0] +2\bz_0 - 4\bx_0, \label{pf:bx 1(2)}\\
	\by_1 & = 2\bx_1 + 2\by_0 - 2\bz_0 + \frac{1}{4}[\bx_1, \by_0] - \frac{1}{4}[\bx_1, \bz_0] - \frac{1}{2}[\by_0, \bz_0], \label{pf:by 1}\\
	\bz_1 & = 2\bx_1 + 2\by_0 - 2\bz_0 + \frac{1}{4}[\bx_1, \by_0] - \frac{1}{4}[\bx_1, \bz_0] + \frac{1}{2}[\by_0, \bz_0], \label{pf:bz 1}
\end{align}
and for $i\geq 2$
\begin{align}
	\bx_2 & = [\bx_1, \by_0] - 2\by_1 + 4\bx_1 + 8\by_0 - 16\bx_0\label{pf:bx 2(1)}\\
		& = [\bx_1, \bz_0] + 2\bz_1 - 4\bx_1 + 8\bz_0 - 16\bx_0,\label{pf:bx 2(2)}\\
	\bx_{i+1} & = [\bx_i, \by_0] - 2\by_i + 4\bx_i + 8\by_{i-1}\label{pf:bx i+1(1)}\\
		& = [\bx_i, \bz_0] + 2\bz_i - 4\bx_i + 8\bz_{i-1},\label{pf:bx i+1(2)}\\
	\by_i & = 2\bx_i + \frac{1}{4}[\bx_i, \by_0] - \frac{1}{4}[\bx_i, \bz_0] - \frac{1}{2}[\by_{i-1}, \bz_0], \label{pf:by i}\\
	\bz_i & = 2\bx_i + \frac{1}{4}[\bx_i, \by_0] - \frac{1}{4}[\bx_i, \bz_0] + \frac{1}{2}[\by_0, \bz_{i-1}]. \label{pf:bz i}
\end{align}
\end{lemma}
\begin{proof}
Lines \eqref{pf:bx 1(1)}, \eqref{pf:bx 1(2)} are obtained from \eqref{eq: [x,y] 0,0}, \eqref{eq:[x,z] 0,0}, respectively.
To get \eqref{pf:by 1}, first solve \eqref{eq: [x,y] 1,0} for $\bx_2$ and use this equation to eliminate $\bx_2$ in \eqref{eq:[x,z] 1,0}. 
Subsequently, simplify the result to get
\begin{equation}\label{pf:[x1,z0]-[x1-y0]}
	[\bx_1,\bz_0] - [\bx_1, \by_0] = 8\bx_1 - 2\by_1 - 2\bz_1 + 8\by_0 - 8 \bz_0.
\end{equation}
Eliminate $\bz_1$ in \eqref{pf:[x1,z0]-[x1-y0]} using \eqref{eq: [y,z] 0,0} and simplify the result.
Then we obtain \eqref{pf:by 1}. 
Moreover, use \eqref{eq: [y,z] 0,0} and \eqref{pf:by 1} to get \eqref{pf:bz 1}. 
Lines \eqref{pf:bx 2(1)}--\eqref{pf:bz i} are similarly obtained.
\end{proof}

\section{Transition Matrices}\label{sec:TM}

Recall the Onsager subalgebra $\mathcal{O}$ of $\tet$ and recall the Lie algebra isomorphism $\sigma: \tet \to L(\sl2)^+$ in Lemma \ref{def: sigma}.
In Propositions \ref{prop:basisOsigma} and \ref{prop:Osigma}, we displayed two bases for $\mathcal{O}^\sigma$:
\begin{align}
	&& && &a_{2i}^\sigma, && b_{2i}^\sigma, && a_{2i+1}^\sigma  && (i\in \mathbb{N}_0), \label{1st basis Osigma} && &&\\
	&& && &\bx_{i}, && \by_{i}, && \bz_{i} &&  (i\in \mathbb{N}_0).\label{2nd basis Osigma} && && 
\end{align}
\noindent
In this section, we discuss the transition matrices between these two bases.
First, we give the transition matrix from \eqref{1st basis Osigma} to \eqref{2nd basis Osigma} in the following lemma.

\begin{lemma}
For $i \in \mathbb{N}_0$ we have
\begin{align}
	\boldsymbol{x}_{i} & = a^\sigma_{2i} + 4 b^\sigma_{2i-2},\label{eq(1):bx sigma}\\
	\boldsymbol{y}_{i} & = a^\sigma_{2i+1} + 2a^\sigma_{2i}+4a^\sigma_{2i-1}-2b^\sigma_{2i},\label{eq(2):bx sigma}\\
	\boldsymbol{z}_{i} & = a^\sigma_{2i+1} + 2a^\sigma_{2i} - 4a^\sigma_{2i-1} + 2b^\sigma_{2i},\label{eq(3):bx sigma}
\end{align}
where $a^\sigma_{-1} :=0$, $b^\sigma_{-2}:=0$.
\end{lemma}
\begin{proof}
Immediate from \eqref{eq:a(i) tet}--\eqref{eq:c(i) tet}.
\end{proof}

\noindent
Next, we give the transition matrix from \eqref{2nd basis Osigma} to \eqref{1st basis Osigma} in the following lemma.

\begin{lemma}
For $i \in \mathbb{N}_0$ we have

\begin{align}
	a^\sigma_{2i} & = \sum^{\lfloor i/2 \rfloor}_{k=0} 4^{2k}\left(\bx_{i-2k} + \by_{i-2k-1} - \bz_{i-2k-1} - 4\by_{i-2k-2} - 4\bz_{i-2k-2}\right), \label{tran:a_2i->x(i)}\\
	b^\sigma_{2i} & = \sum^{\lfloor i/2 \rfloor}_{k=0} 4^{2k}
		\left(-\frac{1}{4}\by_{i-2k} +\frac{1}{4}\bz_{i-2k} -4\bx_{i-2k-1} + \by_{i-2k-1} + \bz_{i-2k-1} \right), \label{tran:b_2i->x(i)} \\
	a^\sigma_{2i+1} &= 
		 \frac{1}{2}\by_{i} + \frac{1}{2}\bz_{i} 
		- 2\sum^{\lfloor i/2 \rfloor}_{k=0} 4^{2k}\left(\bx_{i-2k} + \by_{i-2k-1} - \bz_{i-2k-1} - 4\by_{i-2k-2} - 4\bz_{i-2k-2}\right),\label{tran:a_2i+1->x(i)}
\end{align}
where $\by_{-1}:=0$, $\by_{-2}:=0$, $\bz_{-1}:=0$, $\bz_{-2}:=0$.
\end{lemma}

\begin{proof}
We show \eqref{tran:a_2i->x(i)}. 
For notational convenience, we abbreviate the expression on the right side of \eqref{tran:a_2i->x(i)} as follows.
\begin{equation}\label{lem:eq E(i)}
	\boldsymbol{e}_{i}=\bx_{i} + \by_{i-1} - \bz_{i-1} - 4\by_{i-2} - 4\bz_{i-2} \qquad (i \in \mathbb{N}_0).
\end{equation}
Then, we will show
\begin{equation}\label{pf: Gi=sum Ei}
		a^\sigma_{2i} = \sum^{\lfloor i/2 \rfloor}_{k=0} 4^{2k} \boldsymbol{e}_{i-2k}.
\end{equation}
Evaluate the right-hand side of \eqref{lem:eq E(i)} using \eqref{eq(1):bx sigma}--\eqref{eq(3):bx sigma} and simplify the result to obtain
\begin{equation}\label{pf: eq E(i) (1)}
	\boldsymbol{e}_{i} = a^\sigma_{2i} - 16 a^\sigma_{2i-4},
\end{equation}
where $a^\sigma_{-2}:=0$ and $a^\sigma_{-4}:=0$.
Evaluate the right-hand side of \eqref{pf: Gi=sum Ei} using \eqref{pf: eq E(i) (1)} to obtain $a^\sigma_{2i}$, as desired. 
Therefore, we have shown \eqref{tran:a_2i->x(i)}.
Equation \eqref{tran:b_2i->x(i)} is obtained in a similar manner to \eqref{tran:a_2i->x(i)}.

\noindent
Next, we show \eqref{tran:a_2i+1->x(i)}. 
Adding both sides of equations \eqref{eq(2):bx sigma} and \eqref{eq(3):bx sigma} and solving it for $a^\sigma_{2i+1}$, we have
\begin{equation}\label{eq:a_2i+1 in pf}
	a^\sigma_{2i+1} =  \frac{1}{2}\by_{i} + \frac{1}{2}\bz_{i} - 2a^\sigma_{2i}.
\end{equation}
Use \eqref{tran:a_2i->x(i)} and \eqref{eq:a_2i+1 in pf} to get \eqref{tran:a_2i+1->x(i)}.
The proof is complete.
\end{proof}

We have considered two bases for $\mathcal{O}^\sigma$, as shown in \eqref{1st basis Osigma} and \eqref{2nd basis Osigma}.
We introduce a third basis for $\mathcal{O}^\sigma$, which naturally arises from $\mathcal{O}^\sigma$.
Recall the direct sum \eqref{decomp O sigma} of $\mathcal{O}^\sigma$.

\begin{proposition}\label{lem:basis Delta}
The Lie algebra $\mathcal{O}^\sigma$ has a basis
\begin{equation}\label{eq:basis Delta}
	x\otimes (2t-1)^{i}, \quad y\otimes t(2t-1)^{i}, \quad z \otimes (1-t)(2t-1)^{i} \qquad \qquad (i \in \mathbb{N}_0).
\end{equation}
\end{proposition}
\begin{proof}
One can readily observe that $\{x\otimes (2t-1)^{i}\}_{i \in \mathbb{N}_0}$, $\{y\otimes t(2t-1)^{i}\}_{i \in \mathbb{N}_0}$, and $\{z\otimes (1-t)(2t-1)^{i}\}_{i \in \mathbb{N}_0}$ are bases for $x\otimes \mathbb{F}[t]$, $y\otimes t\mathbb{F}[t]$, and $z\otimes (1-t)\mathbb{F}[t]$, respectively. 
Given these observations and the fact that the sum \eqref{decomp O sigma} is direct, the result follows.
\end{proof}

Next, we discuss a relationship between the basis \eqref{eq:basis Delta} and the other two bases, \eqref{1st basis Osigma} and \eqref{2nd basis Osigma}. 
From Lemma \ref{lem:x,y,z F[t]} we observe
\begin{align*}
	\operatorname{Span}\{\bx_{i} \mid i\in \mathbb{N}_0\} & = \operatorname{Span}\{x\otimes (2t-1)^i \mid i\in \mathbb{N}_0\}, \\
	\operatorname{Span}\{\by_{i} \mid i\in \mathbb{N}_0\} & = \operatorname{Span}\{y\otimes t(2t-1)^i \mid i\in \mathbb{N}_0\}, \\
	\operatorname{Span}\{\bz_{i} \mid i\in \mathbb{N}_0\} & = \operatorname{Span}\{z\otimes (1-t)(2t-1)^i \mid i\in \mathbb{N}_0\}.
\end{align*}

\begin{lemma}
The following {\rm{(i)}}--{\rm{(iii)}} hold.
\begin{enumerate}[label=(\roman*), font=\normalfont]
	\item $X^\sigma_{12}\cap\mathcal{O}^\sigma$ has a basis $\{ x\otimes (2t-1)^i\}_{i \in \mathbb{N}_0}$.
	\item $X^\sigma_{23}\cap\mathcal{O}^\sigma$ has a basis $\{y\otimes t(2t-1)^{i}\}_{i \in \mathbb{N}_0}$.
	\item $X^\sigma_{31}\cap\mathcal{O}^\sigma$ has a basis $\{z\otimes (1-t)(2t-1)^{i}\}_{i \in \mathbb{N}_0}$.
\end{enumerate}
\end{lemma}
\begin{proof}
By Lemma \ref{lem:Lij cap O}.
\end{proof}

We give the transition matrix from \eqref{2nd basis Osigma} to \eqref{eq:basis Delta}:
\begin{lemma}\label{lem:xyz otimes (2t-1)}
For $i\in \mathbb{N}_0$, we have
\begin{align}
	x\otimes (2t-1)^i
	& = \sum^{\lfloor i/2 \rfloor}_{j=0} \frac{\bx_{i-2j}}{(-4)^{i-2j}}, \label{tran:x(i)->pi}\\
	y\otimes t(2t-1)^i
	& = \sum^i_{j=0} \frac{\by_{j}}{(-4)^{j+1}},  \label{tran:y(i)->pi}\\
	z\otimes (1-t)(2t-1)^i
	& = (-1)^{i+1}\sum^i_{j=0} \frac{\bz_{j}}{4^{j+1}}. \label{tran:z(i)->pi}
\end{align}
\end{lemma}
\begin{proof}
We show \eqref{tran:x(i)->pi}.
By Lemma \ref{lem:xij-like elts in L(sl2)+}(i) and \eqref{eq:a(i) tet}, we have
\begin{equation}\label{eq(3): x otimes (2t-1)i}
	x\otimes1=\bx_{0}, \qquad \qquad x\otimes (2t-1)=-\frac{\bx_{1}}{4}, 
\end{equation}
and
\begin{equation}\label{eq(1): x otimes (2t-1)i}
	x\otimes (2t-1)^i  = x\otimes (2t-1)^{i-2} + \frac{1}{(-4)^i}\bx_{i}, \qquad \qquad i=2,3,4,\ldots.
\end{equation}
Evaluate the right-hand side of \eqref{eq(1): x otimes (2t-1)i} and simplify the result to obtain
\begin{equation}\label{eq(2): x otimes (2t-1)i}
x\otimes (2t-1)^i =
\begin{cases}
	\displaystyle  x\otimes 1 + \sum^{(i-2)/2}_{j=0} \frac{1}{(-4)^{i-2j}}\bx_{i-2j} & \qquad \text{if} \quad i=2,4,6,8,\ldots, \\
	\displaystyle  x\otimes (2t-1) + \sum^{(i-3)/2}_{j=0} \frac{1}{(-4)^{i-2j}}\bx_{i-2j} & \qquad \text{if} \quad i=3,5,7,9,\ldots.
\end{cases}
\end{equation}
Combining \eqref{eq(3): x otimes (2t-1)i} and \eqref{eq(2): x otimes (2t-1)i} we obtain \eqref{tran:x(i)->pi}. 
Similarly, we obtain \eqref{tran:y(i)->pi}, \eqref{tran:z(i)->pi}.
\end{proof}

\noindent
We give the transition matrix from \eqref{eq:basis Delta} to \eqref{2nd basis Osigma}:
\begin{lemma}
For $i\in \mathbb{N}_0$, we have
\begin{align}
	\bx_{i} 
	& = 	(-1)^i4^i x\otimes (2t-1)^i + (-1)^{i+1}4^i x\otimes (2t-1)^{i-2}\\
	\by_{i} 
	& = 	(-1)^{i+1}4^{i+1} y \otimes t(2t-1)^{i} + (-1)^{i}4^{i+1} y\otimes  t(2t-1)^{i-1},\\
	\bz_{i}
	& = (-1)^{i+1}4^{i+1} z\otimes (1-t)(2t-1)^i + (-1)^{i+1}4^{i+1} z \otimes (1-t)(2t-1)^{i-1},
\end{align}
where $(2t-1)^{-1} := 0$ and $(2t-1)^{-2} := 0$.
\end{lemma}
\begin{proof}
By Lemma \ref{lem:xij-like elts in L(sl2)+}.
\end{proof}

Next, we give the transition matrix from \eqref{eq:basis Delta} to \eqref{1st basis Osigma}:
\begin{lemma}
For $i\in \mathbb{N}_0$, we have
\begin{align}
	a_{2i}^\sigma 
	&= (-1)^i4^i\left(x\otimes(2t-1)^i + y\otimes t(2t-1)^{i-1} - z \otimes (1-t)(2t-1)^{i-1} \right),\\
	b_{2i}^\sigma 
	&= (-1)^i4^i\left( x\otimes(2t-1)^{i-1} +y\otimes t(2t-1)^i -z\otimes (1-t)(2t-1)^{i} \right), \\
	a_{2i+1}^\sigma 
	&= 2(-1)^{i+1}4^i \left( x\otimes (2t-1)^i + y\otimes t(2t-1)^{i} +z\otimes (1-t)(2t-1)^{i} \right),
\end{align}
where $(2t-1)^{-1} := 0$.
\end{lemma}
\begin{proof}
By Lemma \ref{lem:G,H U_i}.
\end{proof}

\noindent
We give the transition matrix from \eqref{1st basis Osigma} to \eqref{eq:basis Delta}:

\begin{lemma}\label{lem:tran pi -> ai}
For $i\in \mathbb{N}_0$, we have
\begin{align*}
	x\otimes (2t-1)^i
	& = \sum^{\lfloor i/2 \rfloor}_{j=0} \frac{ a^\sigma_{2i-4j} + 4b^\sigma_{2i-4j-2} }{(-4)^{i-2j}} \\
	y\otimes t(2t-1)^i
	& = \sum^i_{j=0} \frac{a^\sigma_{2j+1} + 2a^\sigma_{2j}+4a^\sigma_{2j-1}-2b^\sigma_{2j}}{(-4)^{j+1}},  \\
	z\otimes (1-t)(2t-1)^i
	& = (-1)^{i+1}\sum^i_{j=0} \frac{a^\sigma_{2j+1} + 2a^\sigma_{2j} - 4a^\sigma_{2j-1} + 2b^\sigma_{2j}}{4^{j+1}}. 
\end{align*}
\end{lemma}
\begin{proof}
By Lemma \ref{lem:xyz otimes (2t-1)} along with \eqref{eq:a(i) tet}--\eqref{eq:c(i) tet}.
\end{proof}

We finish this section with a comment.

\begin{proposition}
Recall the basis \eqref{eq:basis Delta} for the Lie algebra $\mathcal{O}^\sigma$.
The Lie bracket acts on this basis as follows.
For $i,j \in \mathbb{N}_0$,
\begin{align}
	[x\otimes (2t-1)^i, x\otimes (2t-1)^j] & = 0, \label{prop Osigma: [x,x]}\\
	[y\otimes t(2t-1)^i, y\otimes t(2t-1)^j] & = 0, \label{prop Osigma: [y,y]}\\
	[z\otimes (1-t)(2t-1)^i, z\otimes (1-t)(2t-1)^j] & = 0. \label{prop Osigma: [z,z]}
\end{align}
Moreover,
\begin{align}
	[x\otimes (2t-1)^i, y\otimes t(2t-1)^j] 
	& = x \otimes (2t-1)^{i+j} + x \otimes (2t-1)^{i+j+1} + 2y \otimes t(2t-1)^{i+j}, \label{prop Osigma: [x,y]}\\
	[x\otimes (2t-1)^i, z\otimes (1-t)(2t-1)^j] 
	& = x \otimes (2t-1)^{i+j+1} - x \otimes (2t-1)^{i+j} - 2z \otimes (1-t)(2t-1)^{i+j},\label{prop Osigma: [x,z]}\\
\begin{split}
	[y\otimes t(2t-1)^i, z\otimes (1-t)(2t-1)^j] 
	& = y \otimes t(2t-1)^{i+j} - y \otimes t(2t-1)^{i+j+1} \\
	& \qquad + z \otimes (1-t)(2t-1)^{i+j} + z \otimes (1-t)(2t-1)^{i+j+1}, \label{prop Osigma: [y,z]}
\end{split}
\end{align}
\end{proposition}
\begin{proof}
Lines \eqref{prop Osigma: [x,x]}--\eqref{prop Osigma: [z,z]} are clear.
We show the equation \eqref{prop Osigma: [x,y]}.
Abbreviate $p_i = (2t-1)^i$ for $i\in \mathbb{N}_0$.
Express the left-hand side of \eqref{prop Osigma: [x,y]} as the vector-valued form in \eqref{eq:vvf}. 
Evaluate this using Lemma \ref{v-val comp}(i) and \eqref{eq:algebra p_i} to get
\begin{equation*}
	\left[\begin{pmatrix} p_i \\ 0 \\ 0 \end{pmatrix}, \begin{pmatrix} 0 \\ tp_j \\ 0 \end{pmatrix}\right]
	= 2\begin{pmatrix} tp_ip_j \\ tp_ip_j \\ 0 \end{pmatrix}
	= \begin{pmatrix} 2tp_{i+j} \\ 2tp_{i+j} \\ 0 \end{pmatrix}
	= \begin{pmatrix} p_{i+j} \\ 0 \\ 0 \end{pmatrix} + \begin{pmatrix} p_{i+j+1} \\ 0 \\ 0 \end{pmatrix} + 2\begin{pmatrix} 0 \\ tp_{i+j} \\ 0 \end{pmatrix},
\end{equation*}
which is equal to the right-hand side of \eqref{prop Osigma: [x,y]}.
We have shown the equation \eqref{prop Osigma: [x,y]}.
Similarly, the equations \eqref{prop Osigma: [x,z]} and \eqref{prop Osigma: [y,z]} follow.
\end{proof}


\appendix
\section{Appendix: some data for $\mathcal{O}'$ and $\mathcal{O}''$}
Recall the Onsager subalgebra $\mathcal{O}$ of $\tet$ generated by $x_{12}$ and $x_{03}$, and recall the automorphisms $\prime=(123)$ and $\prime\prime=(132)$ of $\tet$ from Remark \ref{rmk:O', O''}.
We note that 
\begin{itemize}\itemsep0em 
	\item $\mathcal{O}'$: the Onsager subalgebra of $\tet$ generated by $x_{23}$ and $x_{01}$.
	\item $\mathcal{O}''$: the Onsager subalgebra of $\tet$ generated by $x_{31}$ and $x_{02}$.
\end{itemize}
In this Appendix, we present several data related to $\mathcal{O}'$ and $\mathcal{O}''$.
\begin{lemma} 
Recall the elements $a_i, b_i$ $(i\in \mathbb{N}_0)$ in $\mathcal{O}$ from Definition \ref{def:seq a_i b_i}. 
The actions of $\prime$ and $\prime\prime$ on $a_i, b_i$ are as follows.
\begin{enumerate}[label=(\roman*), font=\normalfont]
	\item $\mathcal{O}':$ 
	\begin{tabular}{ccccccc}
	$\begin{matrix} a'_0  = x_{23}, \\[0.5em] b'_0 = x_{01}, \end{matrix}$ & & &
	$\begin{matrix} a'_{2j-1} = [b'_0, a'_{2j-2}], \\[0.5em] b'_{2j-1} = [a'_0, b'_{2j-2}], \end{matrix}$ & & &
	$\begin{matrix} a'_{2j} = [a'_0, a'_{2j-1}], \\[0.5em] b'_{2j} = [b'_0, b'_{2j-1}] \end{matrix}$ 
	\end{tabular}
	\qquad  $(j=1,2,3,\ldots)$.
	\item $\mathcal{O}'':$ 
	\begin{tabular}{ccccccc}
	$\begin{matrix} a''_0  = x_{31}, \\[0.5em] b''_0 = x_{02}, \end{matrix}$ & & &
	$\begin{matrix} a''_{2j-1} = [b''_0, a''_{2j-2}], \\[0.5em] b''_{2j-1} = [a''_0, b''_{2j-2}], \end{matrix}$ & & &
	$\begin{matrix} a''_{2j} = [a''_0, a''_{2j-1}], \\[0.5em] b''_{2j} = [b''_0, b''_{2j-1}] \end{matrix}$ 
	\end{tabular}
	\qquad  $(j=1,2,3,\ldots)$.
\end{enumerate}
\end{lemma}

\begin{lemma}
Recall the subspace $U_i$ $(i\in \mathbb{N}_0)$ from Definition \ref{def:U_i}.
The following table holds.
\begin{center}
{\renewcommand{\arraystretch}{1.5}
\begin{tabular}{c|ccccc}
	&  & basis for $U'_i$ & & basis for $U''_i$ & \\
	\hline
	$i$: odd & & $a'_i$ & & $a''_i$ & \\
	$i$: even & & $a'_i, b'_i$ & & $a''_i, b''_i$ &
\end{tabular}} \qquad $(i\in \mathbb{N}_0)$.
\end{center}
\end{lemma}

\begin{lemma}
The following (i) and (ii) hold.
\begin{enumerate}[label=(\roman*), font=\normalfont]
\item The Onsager subalgebra $\mathcal{O}'$ satisfies
	\begin{equation*}
		\mathcal{O}' = \sum_{i\in \mathbb{N}_0} U'_i, \qquad
		\qquad (\text{\rm direct sum}).
	\end{equation*}
	Moreover, $\mathcal{O}'$ has a basis $a'_{2i}, b'_{2i}, a'_{2i+1}$ $(i \in \mathbb{N}_0)$.
\item The Onsager subalgebra $\mathcal{O}''$ satisfies
	\begin{equation*}
		\mathcal{O}'' = \sum_{i\in \mathbb{N}_0} U''_i, \qquad
		\qquad (\text{\rm direct sum}).
	\end{equation*}
	Moreover, $\mathcal{O}'$ has a basis $a''_{2i}, b''_{2i}, a''_{2i+1}$ $(i \in \mathbb{N}_0)$.
\end{enumerate}
\end{lemma}

\begin{lemma}\label{A.lem:bxbybz}
Recall the elements $\bx_i,\by_{i},\bz_{i}$ $(i\in \mathbb{N}_0)$ in $\mathcal{O}$ from \eqref{eq:a(i) tet}--\eqref{eq:c(i) tet}.
The actions of $\prime$ and $\prime\prime$ on these elements are as follows.
$$
{\renewcommand{\arraystretch}{1.5}
\begin{tabular}{cc|cc}
	$\text{action of }\prime$ &&& $\text{action of }\prime\prime$ \\
	\hline
	$\begin{array}{l} 	
	\bx_i'  = a'_{2i} + 4 b'_{2i-2}, \\
	\by_i'  = a'_{2i+1} + 2a'_{2i} + 4a'_{2i-1} - 2b'_{2i}, \\
	\bz_i'  = a'_{2i+1} + 2a'_{2i} - 4a'_{2i-1} + 2b'_{2i}, 
	\end{array}$
	&&& $\begin{array}{l} 	
		\bx_i'' = a''_{2i} + 4 b''_{2i-2}, \\
		\by_i'' = a''_{2i+1} + 2a''_{2i} + 4a''_{2i-1} - 2b''_{2i}, \\
		\bz_i'' = a''_{2i+1} + 2a''_{2i} - 4a''_{2i-1} + 2b''_{2i}, 
	\end{array}$ \\
\end{tabular}} 
$$
where $a'_{-1}:=0$, $b'_{-2}:=0$, $a''_{-1}:=0$, $b''_{-2}:=0$.
\end{lemma}

\begin{lemma}The following {\rm(i)}, {\rm(ii)} hold.
\begin{enumerate}[label=(\roman*), font=\normalfont]
	\item The Onsager subalgebra $\mathcal{O}'$ has a basis
	\begin{equation*}
	\bx'_{i},\qquad \by'_{i}, \qquad  \bz'_{i}, \qquad \qquad i\in \mathbb{N}_0.
	\end{equation*}
	\item The Onsager subalgebra $\mathcal{O}''$ has a basis
	\begin{equation*}
	\bx''_{i},\qquad \by''_{i}, \qquad  \bz''_{i}, \qquad \qquad i\in \mathbb{N}_0.
	\end{equation*}	
\end{enumerate}
\end{lemma}

\section{Appendix: the $\sigma$-image of $\mathcal{O}^{\prime}$ and $\mathcal{O}^{\prime\prime}$}

Recall the Lie algebra isomorphism $\sigma: \tet \to L(\sl2)^+$ from Lemma \ref{def: sigma}.

\begin{lemma}
The following {\rm(i)}, {\rm(ii)} hold.
\begin{enumerate}[label=(\roman*), font=\normalfont]
\item Recall a basis $a'_{2i}, b'_{2i}, a'_{2i+1}$ $(i \in \mathbb{N}_0)$ for $\mathcal{O}'$.
	For $i\in \mathbb{N}_0$, the map $\sigma$ sends
	\begin{align*}
		a'_{2i}
		& \quad \longmapsto \quad 
		(-1)^i4^i\begin{pmatrix} (t'-1)(2t'-1)^{i-1} \\ (2t'-1)^i \\ t'(2t'-1)^{i-1} \end{pmatrix}, \\
		b'_{2i}
		& \quad \longmapsto \quad 
		(-1)^i4^i\begin{pmatrix}  (t'-1)(2t'-1)^{i} \\ (2t'-1)^{i-1} \\ t'(2t'-1)^i \end{pmatrix},\\
		a'_{2i+1}
		& \quad \longmapsto \quad 
		2(-1)^{i+1}4^i\begin{pmatrix}  (1-t')(2t'-1)^{i} \\ (2t'-1)^i \\ t'(2t'-1)^{i} \end{pmatrix},
	\end{align*}
	where we interpret $(2t'-1)^{-1}:=0$.
\item Recall a basis $a''_{2i}, b''_{2i}, a''_{2i+1}$ $(i \in \mathbb{N}_0)$ for $\mathcal{O}''$.
	For $i\in \mathbb{N}_0$, the map $\sigma$ sends
	\begin{align*}
		a''_{2i}
		& \quad \longmapsto \quad 
		(-1)^i4^i\begin{pmatrix} t''(2t''-1)^{i-1} \\ (t''-1)(2t''-1)^{i-1} \\ (2t''-1)^i  \end{pmatrix}, \\
		b''_{2i}
		& \quad \longmapsto \quad 
		(-1)^i4^i\begin{pmatrix}  t''(2t''-1)^i \\ (t''-1)(2t''-1)^{i} \\ (2t''-1)^{i-1} \end{pmatrix},\\
		a''_{2i+1}
		& \quad \longmapsto \quad 
		2(-1)^{i+1}4^i \begin{pmatrix}  t''(2t''-1)^{i} \\ (1-t'')(2t''-1)^{i} \\ (2t''-1)^i  \end{pmatrix},
	\end{align*}
	where we interpret $(2t''-1)^{-1}:=0$.	
\end{enumerate}
\end{lemma}

\begin{lemma}\label{B.lem:bxbybz}
The following {\rm(i)}, {\rm(ii)} hold.
\begin{enumerate}[label=(\roman*), font=\normalfont]
\item Recall a basis $\bx'_{i}, \by'_{i}, \bz'_{i}$ $(i \in \mathbb{N}_0)$ for $\mathcal{O}'$.
	For $i\in \mathbb{N}_0$, the map $\sigma$ sends
	\begin{align*}
		\bx'_{i}
		& \quad \longmapsto \quad 
		(-1)^i4^i\begin{pmatrix} 0 \\ (2t'-1)^i - (2t'-1)^{i-2} \\ 0 \end{pmatrix},\\ 
		\by'_{i}
		& \quad \longmapsto \quad 
		(-1)^{i+1}4^{i+1} \begin{pmatrix} 0 \\ 0 \\ t'(2t'-1)^{i} - t'(2t'-1)^{i-1}  \end{pmatrix},\\
		\bz'_{i}
		& \quad \longmapsto \quad 
		(-1)^{i+1}4^{i+1} \begin{pmatrix} (1-t')(2t'-1)^i+(1-t')(2t'-1)^{i-1} \\ 0 \\ 0 \end{pmatrix},
	\end{align*}
	where we interpret $(2t'-1)^{-1}:=0$ and $(2t'-1)^{-2}:=0$.
	The $\sigma$-images of the elements $\bx'_{i}$, $\by'_{i}$, $\bz'_{i}$ $(i\in \mathbb{N}_0)$ form a basis for $\mathcal{O}^{\prime\sigma}$.
\item Recall a basis $\bx''_{i}, \by''_{i}, \bz''_{i}$ $(i \in \mathbb{N}_0)$ for $\mathcal{O}''$.
	For $i\in \mathbb{N}_0$, the map $\sigma$ sends
	\begin{align*}
		\bx''_{i}
		& \quad \longmapsto \quad 
		(-1)^i4^i\begin{pmatrix} 0 \\ 0 \\ (2t''-1)^i - (2t''-1)^{i-2}  \end{pmatrix},\\ 
		\by''_{i}
		& \quad \longmapsto \quad
		(-1)^{i+1}4^{i+1} \begin{pmatrix} t''(2t''-1)^{i} - t''(2t''-1)^{i-1} \\ 0 \\ 0  \end{pmatrix},\\
		\bz''_{i}
		& \quad \longmapsto \quad 
		(-1)^{i+1}4^{i+1} \begin{pmatrix} 0 \\ (1-t'')(2t''-1)^i+(1-t'')(2t''-1)^{i-1} \\ 0 \end{pmatrix},%
	\end{align*}
	where we interpret $(2t''-1)^{-1}:=0$ and $(2t''-1)^{-2}:=0$.	
	The $\sigma$-images of the elements $\bx''_{i}$, $\by''_{i}$, $\bz''_{i}$ $(i\in \mathbb{N}_0)$ form a basis for $\mathcal{O}^{\prime\prime \sigma}$.
\end{enumerate}
\end{lemma}

\begin{corollary}
We have the following tables:
\begin{equation*}
{\renewcommand{\arraystretch}{1.5}
\begin{tabular}{c|cc}
	$X$ & & a basis for $X$ \\
	\hline
	$X^\sigma_{12}\cap \mathcal{O}^{\prime\sigma}$ & & $\{\bz'_{i} \mid i \in \mathbb{N}_0\}$\\
	$X^\sigma_{23}\cap \mathcal{O}^{\prime\sigma}$ & & $\{\bx'_{i} \mid i \in \mathbb{N}_0\}$\\
	$X^\sigma_{31}\cap \mathcal{O}^{\prime\sigma}$ & & $\{\by'_{i} \mid i \in \mathbb{N}_0\}$
\end{tabular},
\qquad \qquad
\begin{tabular}{c|cc}
	$X$ & & a basis for $X$ \\
	\hline
	$X^\sigma_{12}\cap \mathcal{O}^{\prime\prime\sigma}$ & & $\{\by''_{i} \mid i \in \mathbb{N}_0\}$\\
	$X^\sigma_{23}\cap \mathcal{O}^{\prime\prime\sigma}$ & & $\{\bz''_{i} \mid i \in \mathbb{N}_0\}$\\
	$X^\sigma_{31}\cap \mathcal{O}^{\prime\prime\sigma}$ & & $\{\bx''_{i} \mid i \in \mathbb{N}_0\}$
\end{tabular}
}
\end{equation*}
\end{corollary}

\section*{Acknowledgement}
The author expresses his deepest gratitude to Paul Terwilliger for many conversations about this
paper.

\section*{Statements and Declarations}
\textbf{Funding:} The author declares that no funds, grants, or other support were received during
the preparation of this manuscript.

\bigskip
\noindent
\textbf{Competing interests:} The author has no relevant financial or non-financial interests to
disclose.

\bigskip
\noindent
\textbf{Data availability:} No data was used for the research described in the article.



\begin{thebibliography}{00}


\bibitem{2007BenTer} {G. Benkart, P. Terwilliger}
	The universal central extension of the three-point $\mathfrak{sl}_2$ loop algebra
	Proc. Amer. Math. Soc. 135 (2007), no. 6, 1659--1668.

\bibitem{2000DR}{E. Date, S. S. Roan}, 
	The structure of quotients of the Onsager algebra by closed ideals,
	J. Phys. A: Math. Gen. 33 (2000) 3275--3296.

\bibitem{2007Har}{B. Hartwig},
	The tetrahedron algebra and its finite-dimensional irreducible modules,
	Linear Algebra Appl. \textbf{422} (2007), pp. 219--235.

\bibitem{2007HarTer}{B. Hartwig, P. Terwilliger},
	The Tetrahedron algebra, the Onsager algebra, and the $\mathfrak{sl}_2$ loop algebra,
	J. Algebra \textbf{308} (2007), pp. 840--863.



\bibitem{2007ItoTerLAA} {T. Ito, P. Terwilliger},
	Tridiagonal pairs of Krawtchouk type,
	Linear Algebra Appl., \textbf{427} (2007), pp. 218-233.

\bibitem{2008ItoTerCA} {T. Ito, P. Terwilliger},
	Finite-dimensional irreducible modules for the three-point $\mathfrak{sl}_2$ loop algebra,
	Comm. Algebra \textbf{36} (2008), no. 12, 4557--4598.

\bibitem{2022MorLAA} {J.V.S. Morales},
	Linking the special orthogonal algebra $\mathfrak{so}_4$ and the tetrahedron algebra $\boxtimes$,
	Linear Algebra Appl. \textbf{637} (2022), 212–239.

\bibitem{2014MorPas} {J.V.S. Morales, A. Pascasio},
	An action of the tetrahedron algebra on the standard module for the Hamming graphs and Doob graphs,
	Graphs Combin. \textbf{30} (2014), no. 6, 1513--1527.


\bibitem{1944Ons}{L. Onsager}, 
	Crystal statistics. I. A two-dimensional model with an order-disorder transition,
	Phys. Rev. \textbf{65} (1944) 117--149.

\bibitem{1989Per}{J. H. H. Perk},
	Star-triangle relations, quantum Lax pairs, and higher genus curves, 
	Proceedings of Symposia in Pure Mathematics \textbf{49} 341--354. Amer. Math. Soc., Providence, RI, 1989.




\end{thebibliography}
\end{document}